\documentclass{siamltex}%
\usepackage{amsfonts}
\usepackage{amsmath}
\usepackage{amssymb}
\usepackage{graphicx}
\usepackage{hyperref}%
\newtheorem{remark}[theorem]{Remark}

\begin{document}

\title{Entropy for practical stabilization\thanks{B. H. thanks the European
Commission for funding through the Marie Curie fellowship STALDYS792919
(Statistical Learning for Dynamical Systems).}}
\author{Fritz Colonius\thanks{Institut f\"{u}r Mathematik, Universit\"{a}t Augsburg,
Augsburg, Germany}
\and Boumediene Hamzi\thanks{Department of Mathematics, Imperial College London,
London, UK}}
\maketitle

\begin{abstract}
For deterministic continuous time nonlinear control systems, $\varepsilon
$-practical stabilization entropy and practical stabilization entropy are
introduced. Here the rate of attraction is specified by a $\mathcal{KL}%
$-function. Upper and lower bounds for the diverse entropies are proved, with
special attention to exponential $\mathcal{KL}$-functions. The relation to
feedbacks is discussed, the linear case and several nonlinear examples are
analyzed in detail.

\end{abstract}

\today

\textbf{Keywords:} practical stabilization, practical stabilization entropy,
invariance entropy

\textbf{MSC\ 2010: }93D15, 37B40, 94A17.

\section{Introduction}

This paper analyzes entropy properties for practical stabilization of control
systems described by ordinary differential equations. The constructions are
similar to the theory of invariance entropy motivated by digitally connected
control systems. A controller which in finite time intervals $[0,\tau]$
receives only finitely many data can generate only finitely many (open-loop)
control functions on $[0,\tau]$. Invariance entropy abstracts from this
situation by counting the number of control functions needed in order to
achieve invariance on $[0,\tau]$ and then looks at the exponential growth rate
of this number as $\tau$ tends to infinity. The specific relations to minimal
data rates are worked out in the monograph Kawan \cite{Kawa13}, where also the
relations to feedback entropy introduced in the pioneering work Nair, Evans,
Mareels and Moran \cite{NEMM04} are clarified. Related work includes Kawan and
Da Silva \cite{KawaDS16} using hyperbolicity conditions, Huang and Zhong
\cite{HuanZ18} for a dimension-like characterization, and Wang, Huang, and Sun
\cite{WangHS19} for a measure-theoretic version. A similar approach is taken
in Colonius \cite{Colo12b} for entropy of exponential stabilization and
analogous constructions are also used for state estimation in Liberzon and
Mitra \cite{LibeM18} as well as Matveev and Pogromsky \cite{MatvP16, MatvP19},
Kawan, Matveev and Pogromsky \cite{KaMP21}. Related work is also due to Berger
and Jungers \cite{BeJu20}, where finite-data rates for linear systems with
switching are analyzed.

Our motivation to consider practical stabilizability is twofold: There are
systems where instead of stabilization only practical stabilization is
possible (throughout the paper \emph{stability/stabilization} means
\emph{asymptotic stability/stabilization} if not indicated otherwise). Some
examples are provided in Section \ref{Section5}. Perhaps more relevant is the
fact that standard stabilization algorithms may only lead to practical
stabilization although, theoretically, stabilization is possible. This is the
case for Economic Model Predictive Control (EMPC) schemes, cf. Zanon and
Faulwassser \cite{ZF18} where practical stabilization is achieved, but
stabilization does not hold \cite[Theorem 1 and Theorem 3]{ZF18}. Furthermore,
sampled feedback of stabilizable systems may only lead to practical
stabilization, cf. Gr\"{u}ne \cite[Section 9.4]{Grue} for a simple example.
Similarly, the restriction to other classes of feedbacks may entail that only
practical stabilization is possible.

The purpose of the present paper is to contribute to an understanding of
entropy for practical stabilization. We specify not only the sets $\Gamma$ of
initial states and the \textquotedblleft target set\textquotedblright%
\ $\Lambda$ of final states, but also the convergence rate given by a
$\mathcal{KL}$-function, similarly as in the definition of practical stability
in Hamzi and Krener \cite{HamzK03}. For this purpose we introduce new notions
of entropy for $\varepsilon$-practical stabilization, practical stabilization,
and also for stabilization. We consider the minimal number of control
functions needed in order to achieve the practical stabilization goal on a
finite time interval. Then we let time tend to infinity and consider the
exponential growth rate of these numbers. This is similar to the familiar
definition of invariance entropy as exposed, in particular, in Kawan
\cite{Kawa13}. Note, however, that here in contrast to \cite{Kawa13}, the set
of initial states is in general not a subset of the target set. The relation
to feedbacks is briefly discussed based on a new notion of entropy for feedbacks.

In more specific terms the basic construction for entropy is the following.
Consider a control system in $\mathbb{R}^{d}$ of the form $\dot{x}%
(t)=f(x(t),u(t))$ with a set $\mathcal{U}$ of admissible control functions $u$
and trajectories denoted by $\varphi(t,x_{0},u),t\geq0$. Fix subsets
$\Gamma,\Lambda\subset\mathbb{R}^{d}$ and a $\mathcal{KL}$-function $\zeta$.
For $\tau>0$ a set $\mathcal{S}\subset\mathcal{U}$ of controls is $(\tau
,\zeta,\Gamma,\Lambda)$-spanning if for every initial value $x_{0}\in\Gamma$
there exists $u\in\mathcal{S}$ with%
\begin{equation}
d(\varphi(t,x_{0},u),\Lambda)\leq\zeta\left(  d(x_{0},\Lambda),t\right)
\text{ for all }t\in\lbrack0,\tau]. \label{01}%
\end{equation}
Denoting by $r(\tau,\zeta,\Gamma,\Lambda)$ the minimal cardinality of a
$(\tau,\zeta,\Gamma,\Lambda)$-spanning set we define the stabilization entropy
by $\overline{\lim}_{\tau\rightarrow\infty}\frac{1}{\tau}\log r(\tau
,\zeta,\Gamma,\Lambda)$. This number measures, how fast the average number of
required controls increases, when the system should approach the set $\Lambda$
with the bound (\ref{01}) as time $\tau$ tends to infinity. In order to
guarantee the existence of finite spanning sets of controls, this notion has
to be slightly modified. Practical stabilizability properties, which are in
the focus of the present paper, are obtained if we require that the solutions
approach $\Lambda$ only approximately, cf. Definition \ref{definition_main}.
We remark that the construction of entropy via spanning sets follows the
classical construction of entropy for dynamical systems in metric spaces due
to Bowen and Dinaburg. The logarithm with base $2$ is directly related to the
number of bits needed to choose a control $u$; for continuous time systems, as
considered here, the natural logarithm is more convenient.

The contents of this paper are as follows. Section \ref{Section2} introduces
$\varepsilon$-practical stabilization entropy, practical stabilization
entropy, and stabilization entropy about a compact set $\Lambda$ for compact
sets $\Gamma$ of initial states and compact control ranges. Also modifications
for non-compact control ranges and non-compact sets of initial states are
indicated. Section \ref{Section3} proves upper bounds for the diverse
entropies, and lower bounds based on volume growth arguments are established.
Special attention is given to exponential $\mathcal{KL}$-functions of the form
$\zeta(r,s)=e^{-\alpha s}Mr,r,s\geq0$ with $\alpha>0$ and $M\geq1$. Section
\ref{Section4} briefly discusses the relation to feedbacks. Section
\ref{Section5} analyzes linear systems, and two scalar nonlinear examples
where only practically stabilizing quadratic and piecewise linear feedbacks,
resp., can be constructed (the corresponding proofs are given in an appendix).
For these systems and a similar system in $\mathbb{R}^{d}$ estimates for the
entropies can be obtained. The analysis reveals some subtleties in the
constructions. Finally, Section \ref{Section6} draws some conclusions and
presents open questions.

\textbf{Notation.} A $\mathcal{KL}$-function is a continuous function
$\zeta:[0,\infty)\times\lbrack0,\infty)\rightarrow\lbrack0,\infty)$ such that
$\zeta(r,s)$ is strictly increasing in $r$ for fixed $s$ with $\zeta(0,s)=0$
and strictly decreasing with respect to $s$ for fixed $r$ with $\lim
_{s\rightarrow\infty}\zeta(r,s)=0$. In a metric space, the distance of a point
$x$ to a nonvoid set $A$ is $d(x,A):=\inf\{d(x,a)\left\vert a\in A\right.  \}$
and for a compact set $A$ the $\varepsilon$-neighborhood is $N(A;\varepsilon
)=\{x\left\vert d(x,A)<\varepsilon\right.  \}$. For a point $a$ we write the
ball with radius $\varepsilon$ around $a$ as $\mathbf{B}(a,\varepsilon
)=\{x\left\vert d(x,a)<\varepsilon\right.  \}$. The cardinality, viz. the
number of elements, of a finite set $A$ is denoted by $\#A$. The natural
logarithm is denoted by $\log$ and the limit superior is denoted by
$\overline{\lim}$.

\section{Entropy notions\label{Section2}}

We consider control systems in $\mathbb{R}^{d}$ of the form%
\begin{equation}
\dot{x}(t)=f(x(t),u(t)),u(t)\in U. \label{2.1}%
\end{equation}
The control range $U$ is a subset of $\mathbb{R}^{m}$ and the set of
admissible~control functions is given by
\[
\mathcal{U}=\{u\in L^{\infty}([0,\infty),\mathbb{R}^{m}):u(t)\in U\text{ for
almost all }t\}.
\]
We assume standard conditions on $f$ guaranteeing existence and uniqueness of
solutions $\varphi(t,x_{0},u),t\geq0$, with $\varphi(0,x_{0},u)=x_{0}$ for all
$x_{0}\in\mathbb{R}^{d}$ and $u\in\mathcal{U}$ as well as continuous
dependence on initial values.

Consider the following stability properties for a differential equation
$\dot{x}=g(x,t)$ with (unique) solutions $\psi(t,x_{0}),t\geq0$, for initial
conditions $\psi(0,x_{0})=x_{0}$.

\begin{definition}
\label{Definition2.1}Let $\Gamma,\Lambda\subset\mathbb{R}^{d}$ and let $\zeta$
be a $\mathcal{KL}$-function. Then the system is $(\zeta,\Gamma,\Lambda
)$-stable, if every $x_{0}\in\Gamma$ satisfies%
\[
d(\psi(t,x_{0}),\Lambda)\leq\zeta(d(x_{0},\Lambda),t)\text{ for }t\geq0.
\]
For $\varepsilon>0$ it is $\varepsilon$-practically $(\zeta,\Gamma,\Lambda
)$-stable if every $x_{0}\in\Gamma$ satisfies%
\[
d(\psi(t,x_{0}),\Lambda)\leq\zeta(d(x_{0},\Lambda),t)+\varepsilon\text{ for
}t\geq0.
\]

\end{definition}

Recall that stability of an equilibrium $e$ is equivalent to the existence of
a $\mathcal{KL}$-function for $\Lambda=\{e\}$, cf. Clarke, Ledyaev and Stern
\cite[Lemma 2.6]{CLS98}. We are interested in the data rate needed to make the
control system (\ref{2.1}) stable or at least $\varepsilon$-practically stable
for certain $\varepsilon>0$ or for all $\varepsilon>0$. Again, let subsets
$\Gamma,\Lambda\subset\mathbb{R}^{d}$ and a $\mathcal{KL}$-function $\zeta$ be
given. For $\tau,\varepsilon>0$ we call a set $\mathcal{S}\subset\mathcal{U}$
of controls $(\tau,\varepsilon,\zeta,\Gamma,\Lambda)$-spanning if for every
$x_{0}\in\Gamma$ there exists $u\in\mathcal{S}$ with%
\begin{equation}
d(\varphi(t,x_{0},u),\Lambda)\leq\zeta\left(  d(x_{0},\Lambda)+\varepsilon
,t\right)  \text{ for all }t\in\lbrack0,\tau]. \label{span1}%
\end{equation}
(Here $\varepsilon>0$ is introduced in order to guarantee that finite spanning
sets exist, cf. Remark \ref{Remark2.3}.) Furthermore, a set $\mathcal{S}%
\subset\mathcal{U}$ of controls is called practically $(\tau,\varepsilon
,\zeta,\Gamma,\Lambda)$-spanning if for every $x_{0}\in\Gamma$ there exists
$u\in\mathcal{S}$ with%
\begin{equation}
d(\varphi(t,x_{0},u),\Lambda)\leq\zeta\left(  d(x_{0},\Lambda)+\varepsilon
,t\right)  +\varepsilon\text{ for all }t\in\lbrack0,\tau]. \label{span2}%
\end{equation}
The minimal cardinality of a $(\tau,\varepsilon,\zeta,\Gamma,\Lambda
)$-spanning set and a practically $(\tau,\varepsilon,\zeta,\Gamma,\Lambda
)$-spanning set are denoted by $r_{\mathrm{s}}(\tau,\varepsilon,\zeta
,\Gamma,\Lambda)$ and $r_{\mathrm{ps}}(\tau,\varepsilon,\zeta,\Gamma,\Lambda
)$, resp. If there is no finite spanning set or no spanning set at all, we set
these numbers equal to $+\infty$.

\begin{definition}
\label{definition_main}Let $\Gamma,\Lambda\subset\mathbb{R}^{d}$ and let
$\zeta$ be a $\mathcal{KL}$-function.

(i) For $\varepsilon>0$ the $\varepsilon$-stabilization entropy and the
stabilization entropy are%
\[
h_{\mathrm{s}}(\varepsilon,\zeta,\Gamma,\Lambda)=\underset{\tau\rightarrow
\infty}{\overline{\lim}}\frac{1}{\tau}\log r_{\mathrm{s}}(\tau,\varepsilon
,\zeta,\Gamma,\Lambda)\text{ and }h_{\mathrm{s}}(\zeta,\Gamma,\Lambda
)=\lim_{\varepsilon\rightarrow0}h_{\mathrm{s}}(\varepsilon,\zeta
,\Gamma,\Lambda).
\]

(ii) For $\varepsilon>0$ the\textit{ }$\varepsilon$-\textit{practical
stabilization entropy} $h_{\mathrm{ps}}(\varepsilon,\zeta,\Gamma,\Lambda)$ and
the practical stabilization entropy $h_{\mathrm{ps}}(\zeta,\Gamma,\Lambda)$
are
\begin{equation}
h_{\mathrm{ps}}(\varepsilon,\zeta,\Gamma,\Lambda)=\underset{\tau
\rightarrow\infty}{\overline{\lim}}\frac{1}{\tau}\log r_{\mathrm{ps}}%
(\tau,\varepsilon,\zeta,\Gamma,\Lambda)\text{ and }h_{\mathrm{ps}}%
(\zeta,\Gamma,\Lambda)=\lim_{\varepsilon\rightarrow0}h_{\mathrm{ps}%
}(\varepsilon,\zeta,\Gamma,\Lambda)\text{.} \label{ineq2}%
\end{equation}

\end{definition}

For $\varepsilon_{1}>\varepsilon_{2}>0$ the following inequalities are easily
seen:%
\begin{equation}
h_{\mathrm{ps}}(\varepsilon_{1},\zeta,\Gamma,\Lambda)\leq h_{\mathrm{ps}%
}(\varepsilon_{2},\zeta,\Gamma,\Lambda)\leq h_{\mathrm{ps}}(\zeta
,\Gamma,\Lambda)\leq h_{\mathrm{s}}(\zeta,\Gamma,\Lambda). \label{ineq1}%
\end{equation}
This shows, in particular, that in (\ref{ineq2}) the limit for $\varepsilon
\rightarrow0$ exists and coincides with the supremum over $\varepsilon>0$ (it
may equal $+\infty$). Our results will mainly concern compact control ranges
$U$ and compact sets $\Gamma$ and $\Lambda$, but cf. Remarks
\ref{Remark_unboundedU} and \ref{unbounded_gamma} for generalizations.
Definition \ref{definition_main} does not require that $\Gamma$ is a
neighborhood of $\Lambda$. Nevertheless, situations where $\Gamma$ is a
neighborhood of $\Lambda$ or, at least, has nonvoid interior are certainly
most interesting.

First we will ascertain that under weak assumptions finite spanning sets exist.

\begin{lemma}
\label{Lemma_basic}Consider for a control system of the form (\ref{2.1})
subsets $\Gamma,\Lambda\subset\mathbb{R}^{d}$ and a $\mathcal{KL}$-function
$\zeta$. Assume further that $\Gamma$ is compact and fix $\varepsilon>0$.

(i) Suppose that for every $x_{0}\in\Gamma$ there is a control $u\in
\mathcal{U}$ with%
\[
d(\varphi(t,x_{0},u),\Lambda)<\zeta(d(x_{0},\Lambda)+\varepsilon,t)\text{ for
all }t\geq0.
\]
Then for every $\tau>0$ there is a finite set $\mathcal{S}=\{u_{1}%
,\ldots,u_{n}\}\subset\mathcal{U}$ such that for every $x_{0}\in\Gamma$ there
is $u_{j}\in\mathcal{S}$ with%
\[
d(\varphi(t,x_{0},u_{j}),\Lambda)<\zeta\left(  d(x_{0},\Lambda)+2\varepsilon
,t\right)  \text{ for all }t\in\lbrack0,\tau].
\]

(ii) Suppose that for every $x_{0}\in\Gamma$ there is a control $u\in
\mathcal{U}$ with%
\[
d(\varphi(t,x_{0},u),\Lambda)<\zeta(d(x_{0},\Lambda)+\varepsilon
,t)+\varepsilon\text{ for all }t\geq0.
\]
Then for every $\tau>0$ there is a finite set $\mathcal{S}=\{u_{1}%
,\ldots,u_{n}\}\subset\mathcal{U}$ such that for every $x_{0}\in\Gamma$ there
is $u_{j}\in\mathcal{S}$ with%
\[
d(\varphi(t,x_{0},u_{j}),\Lambda)<\zeta\left(  d(x_{0},\Lambda)+2\varepsilon
,t\right)  +2\varepsilon\text{ for all }t\in\lbrack0,\tau].
\]

\end{lemma}

\begin{proof}
(i) For every $x_{0}\in\Gamma$ choose a control $u\in\mathcal{U}$ with%
\[
d(\varphi(t,x_{0},u),\Lambda)<\zeta(d(x_{0},\Lambda)+\varepsilon,t)\text{ for
all }t\in\lbrack0,\tau].
\]
By continuous dependence on initial values (as assumed for (\ref{2.1})) there
is $\delta$ with $0<\delta<\varepsilon$ such that for all $x_{1}\in
\mathbb{R}^{d}$ with $\left\Vert x_{0}-x_{1}\right\Vert <\delta$ and for all
$t\in\lbrack0,\tau]$%
\begin{align*}
d(\varphi(t,x_{1},u),\Lambda)  &  <\zeta(d(x_{0},\Lambda)+\varepsilon
,t)\leq\zeta(\left\Vert x_{0}-x_{1}\right\Vert +d(x_{1},\Lambda)+\varepsilon
,t)\\
&  <\zeta(\delta+d(x_{1},\Lambda)+\varepsilon,t)<\zeta(d(x_{1},\Lambda
)+2\varepsilon,t).
\end{align*}
Here we have used that for all $x,y\in\mathbb{R}^{d}$ one has $d(x,\Lambda
)\leq\left\Vert x-y\right\Vert +d(y,\Lambda)$ together with the monotonicity
properties of the $\mathcal{KL}$-function $\zeta$. Now compactness of $\Gamma$
shows that there is a finite set $\mathcal{S}=\{u_{1},\ldots,u_{n}%
\}\subset\mathcal{U}$ such that for each $x_{1}\in\Gamma$ there is $u_{j}%
\in\mathcal{S}$ satisfying for all $t\in\lbrack0,\tau]$%
\[
d(\varphi(t,x_{1},u_{j}),\Lambda)<\zeta\left(  d(x_{1},\Lambda)+2\varepsilon
,t\right)  .
\]
(ii) This is proved analogously.
\end{proof}

We observe that the assumption in Lemma \ref{Lemma_basic}(i) holds, in
particular, if there exists a stabilizing feedback. See also Section
\ref{Section4} for more on the relations to feedbacks.

The following remarks refer to related results in the literature and elaborate
on some variants of the entropy notions introduced in Definition
\ref{definition_main}.

\begin{remark}
\label{Remark2.3}In Colonius \cite{Colo12b} the following notion of entropy
for exponential stabilization about an equilibrium in the origin is
introduced. Consider a compact set $\Gamma\subset\mathbb{R}^{d}$ of initial
states, and let $\alpha>0,M>1$, and $\varepsilon>0$. For a time $\tau>0$ a
subset $\mathcal{S}\subset\mathcal{U}$ is called $(\tau,\varepsilon
,\alpha,M,\Gamma)$-spanning if for all $x_{0}\in\Gamma$ there is
$u\in\mathcal{S}$ with
\begin{equation}
\left\Vert \varphi(t,x_{0},u)\right\Vert <e^{-\alpha t}(\varepsilon
+M\left\Vert x_{0}\right\Vert )\text{ for all }t\in\lbrack0,\tau].
\label{spanning0}%
\end{equation}
The minimal cardinality of such a set is denoted by $s_{\mathrm{stab}}%
(\tau,\varepsilon,\alpha,M,\Gamma)$ and the stabilization entropy is defined
by%
\[
h_{\mathrm{stab}}(\alpha,M,\Gamma)=\lim_{\varepsilon\rightarrow0}%
\underset{\tau\rightarrow\infty}{\overline{\lim}}\frac{1}{\tau}\log
s_{\mathrm{stab}}(\tau,\varepsilon,\alpha,M,\Gamma).
\]
The spanning condition (\ref{spanning0}) can be rewritten in the following
way: With $\Lambda=\{0\}$ let a $\mathcal{KL}$-function $\zeta$ be defined by
$\zeta(r,s):=e^{-\alpha s}Mr$. With $M\varepsilon$ instead of $\varepsilon$,
condition (\ref{spanning0}) is%
\[
d(\varphi(t,x_{0},u),\{0\})<e^{-\alpha t}M(\varepsilon+\left\Vert
x_{0}\right\Vert )=\zeta(d(x_{0},\{0\})+\varepsilon,t)\text{ for all }%
t\in\lbrack0,\tau].
\]
Thus $h_{\mathrm{stab}}(\alpha,M,\Gamma)$ is a special case of
\textit{stabilization entropy as specified in Definition \ref{definition_main}%
(i). In }\cite[Proposition 2.2]{Colo12b}) it is shown that finite spanning
sets can only be expected for positive $\varepsilon>0$ in (\ref{spanning0}).
This is the reason why we also consider positive $\varepsilon$ in Definition
\ref{definition_main}(i). Observe that one could similarly relax the condition
for $\varepsilon$-practical stability in Definition \ref{Definition2.1} by
requiring%
\[
d(\psi(t,x_{0}),\Lambda)\leq\zeta(d(x_{0},\Lambda)+\varepsilon,t)+\varepsilon
\text{ for }t\geq0.
\]
Relations of entropy to minimal bit rates for (non-exponential) stabilization
are given in \cite[Lemma 5.2 and Theorem 5.3]{Colo12b}.
\end{remark}

\begin{remark}
\label{Remark4}Hamzi and Krener \cite{HamzK03} call a control system locally
practically stabilizable around an equilibrium in the origin if for every
$\varepsilon>0$ there exists an open set $D$ containing the closed ball
$\mathbf{B}(0,\varepsilon)$, a $\mathcal{KL}$-function $\zeta_{\varepsilon}$,
a positive constant $\delta=\delta(\varepsilon)$ and a control law
$u=k_{\varepsilon}(x)$ such that for any initial value $x(0)$ with $\left\Vert
x(0)\right\Vert <\delta$, the solution $x(t)$ of the feedback system $\dot
{x}=f(x,k_{\varepsilon}(x))$ exists and satisfies
\begin{equation}
d(x(t),\mathbf{B}(0,\varepsilon))\leq\zeta_{\varepsilon}\left(
d(x(0),\mathbf{B}(0,\varepsilon)),t\right)  \text{ for all }t\geq0.
\label{HK_stab}%
\end{equation}
Note that here $\delta(\varepsilon)<\varepsilon$ is admitted, hence
attractivity is not required. The trajectories may leave $\mathbf{B}%
(0,\varepsilon)$, but the bound (\ref{HK_stab}) ensures that they converge to
it for $t\rightarrow\infty$. In view of the fact, that here the $\mathcal{KL}%
$-function $\zeta_{\varepsilon}$ depends on $\varepsilon$ the following
variant of practical stabilization entropy might be considered: Define for
$\varepsilon>0$ the\textit{ }$\varepsilon$-\textit{practical stabilization
entropy} by%
\[
\hat{h}_{\mathrm{ps}}(\varepsilon,\Gamma,\Lambda)=\inf_{\zeta_{\varepsilon}%
}\underset{\tau\rightarrow\infty}{\overline{\lim}}\frac{1}{\tau}\log
r_{\mathrm{ps}}(\tau,\varepsilon,\zeta_{\varepsilon},\Gamma,\Lambda)\text{,}%
\]
where the infimum is taken over all $\mathcal{KL}$-functions $\zeta
_{\varepsilon}$, and let a\ practical stabilization entropy be $\hat
{h}_{\mathrm{ps}}(\Gamma,\Lambda)=\lim_{\varepsilon\rightarrow0}\hat
{h}_{\mathrm{ps}}(\varepsilon,\Gamma,\Lambda)$. Also a corresponding local
version might be introduced by requiring the practical $(\tau,\varepsilon
,\zeta,\Gamma,\Lambda)$-spanning condition (\ref{span2}) only for all
$x_{0}\in\Gamma_{\varepsilon}$, where $\Gamma_{\varepsilon}$ is compact
neighborhood of $\Lambda$. If one requires this spanning condition for all
initial values $x_{0}$ in a compact neighborhood $\Gamma_{\delta(\varepsilon
)}$ of $\Lambda$ containing a $\delta(\varepsilon)$-neighborhood of $\Lambda$,
$\delta(\varepsilon)>0$, one obtains a local notion without attractivity requirement.
\end{remark}

\begin{remark}
For control-affine systems, Da Silva and Kawan define in \cite{DaSiK18a} a
version of invariance entropy (for \textquotedblleft practical
stabilization\textquotedblright) in the special situation, where (in our
notation) $\Gamma=\Lambda$ is a compact subset of a control set $D$ with
nonvoid interior (i.e., a maximal set with approximate controllability). Then
they consider the maximum of the corresponding entropies taken over all
$\Gamma=\Lambda$ contained in $D$. Under a uniform hyperbolicity condition for
$\mathrm{cl}D$, \cite[Theorem 9]{DaSiK18a} shows that the corresponding
entropy varies continuously with respect to system parameters.
\end{remark}

\begin{remark}
\label{Remark_unboundedU}In the examples in Subsections \ref{Subsection5.2}
and \ref{Subsection5.3} (cf. Theorem \ref{Theorem5.3} and Theorem
\ref{Theorem_unbounded}) also unbounded closed control ranges $U$ occur, where
it will be appropriate to employ a reduction to compact control ranges by
using the following modified notion: For compact sets $K\subset\mathbb{R}^{m}$
a subset $\mathcal{S}\subset\mathcal{U}$ of controls with values in $U\cap K$
is practically $(\tau,\varepsilon,\zeta,\Gamma,\Lambda,U\cap K)$-spanning if
for every $x_{0}\in\Gamma$ there exists $u\in\mathcal{S}$ with%
\[
d(\varphi(t,x_{0},u),\Lambda)\leq\zeta\left(  d(x_{0},\Lambda)+\varepsilon
,t\right)  +\varepsilon\text{ for all }t\in\lbrack0,\tau].
\]
Denoting the minimal cardinality of such a spanning set by $r_{\mathrm{ps}%
}(\tau,\varepsilon,\zeta,\Gamma,\Lambda,U\cap K)$ we define the\textit{
}$\varepsilon$-\textit{practical stabilization entropy} by
\[
h_{\mathrm{ps}}(\varepsilon,\zeta,\Gamma,\Lambda,U)=\inf_{K}\underset
{\tau\rightarrow\infty}{\overline{\lim}}\frac{1}{\tau}\log r_{\mathrm{ps}%
}(\tau,\varepsilon,\zeta,\Gamma,\Lambda,U\cap K)\text{,}%
\]
where the infimum is taken over all compact subsets $K\subset\mathbb{R}^{m}$.
Then the practical stabilization entropy again is obtained by letting
$\varepsilon\rightarrow0$. In the case of an exponential $\mathcal{KL}%
$-function $\zeta(r,s)=e^{-\alpha s}Mr$, the relevant quantity is the
exponential rate $\alpha$. In the examples in Subsections \ref{Subsection5.2}
and \ref{Subsection5.3} constants $M$ which depend on $\varepsilon$ occur
while $\alpha$ does not.
\end{remark}

\begin{remark}
\label{unbounded_gamma}In the theory developed below, compactness of the set
$\Gamma$ of initial states plays a crucial role. For general closed sets
$\Gamma$ a reasonable notion of practical stabilization entropy might be
introduced as $h_{\mathrm{ps}}(\zeta,\Gamma,\Lambda):=\sup_{K}h_{\mathrm{ps}%
}(\zeta,\Gamma\cap K,\Lambda)$, where the supremum is taken over all compact
sets $K\subset\mathbb{R}^{d}$. This is in the same spirit as the definition of
topological entropy for uniformly continuous maps on metric spaces, cf.
Walters \cite[Definition 7.10]{Walt82}.
\end{remark}

\section{Bounds for practical stabilization entropy\label{Section3}}

In this section we derive upper and lower bounds for the practical
stabilization entropy and the stabilization entropy.

First we present an upper bound for the $\varepsilon$-practical stabilization
entropy. The proof is based on a cut-off function and is a modification of the
proofs in Katok and Hasselblatt \cite[Theorem 3.3.9]{KatH95} (for topological
entropy) as well as Colonius and Kawan \cite[Theorem 4.2]{ColoK09a} (for
invariance entropy).

For compact sets $\Gamma,\Lambda\subset\mathbb{R}^{d}$, a $\mathcal{KL}%
$-function $\zeta$, and $\varepsilon\geq0$ define the compact set%
\[
P_{\varepsilon}:=\left\{  x\in\mathbb{R}^{d}\left\vert d(x,\Lambda)\leq
\zeta(\max_{y\in\Gamma}d(y,\Lambda)+\varepsilon,0)+\varepsilon\right.
\right\}
\]
and define the constant $L_{\varepsilon}:=\max_{(x,u)\in P_{\varepsilon}\times
U}\left\Vert f_{x}(x,u)\right\Vert <\infty$ where $f_{x}(x,u)=\frac{\partial
f}{\partial x}(x,u)$ and $U$ is compact. Observe that $L_{\varepsilon}$
depends on $\zeta$ and $\Gamma$ and, naturally, on $\Lambda$.

\begin{theorem}
\label{Theorem_upper}Consider for control system (\ref{2.1}) compact sets
$\Gamma,\Lambda\subset\mathbb{R}^{d}$ and let $\zeta$ be $\mathcal{KL}%
$-function. Suppose that the control range $U$ is compact, that $f$ is
continuous and $f$ is differentiable with respect to $x$ and the partial
derivative $f_{x}(x,u)$ is continuous in $(x,u)$.

(i) Fix $\varepsilon>0$. If for every $x_{0}\in\Gamma$ there is a control
$u\in\mathcal{U}$ with%
\begin{equation}
d(\varphi(t,x_{0},u),\Lambda)<\zeta(d(x_{0},\Lambda)+\varepsilon
,t)+\varepsilon\text{ for all }t\geq0, \label{ass2}%
\end{equation}
then the $2\varepsilon$-practical stabilization entropy satisfies
$h_{\mathrm{ps}}(2\varepsilon,\zeta,\Gamma,\Lambda)\leq L_{\varepsilon}d$.

(ii) If the assumption in (i) hold for all $\varepsilon>0$, the
\textit{practical stabilization entropy} satisfies $h_{\mathrm{ps}}%
(\zeta,\Gamma,\Lambda)\leq L_{0}d$.
\end{theorem}

\begin{proof}
Define for $\varepsilon\geq0$%
\[
R_{\varepsilon}:=\{(x_{0},u)\in\Gamma\times\mathcal{U}\left\vert
d(\varphi(t,x_{0},u),\Lambda)\leq\zeta(d(x_{0},\Lambda)+\varepsilon
,t)+\varepsilon\text{ for all }t\geq0\right.  \}.
\]
Note that every $(x_{0},u)\in R_{\varepsilon}$ satisfies $\varphi
(t,x_{0},u)\in P_{\varepsilon}$ for $t\geq0$, since%
\[
d(\varphi(t,x_{0},u),\Lambda)\leq\zeta(d(x_{0},\Lambda)+\varepsilon
,t)+\varepsilon\leq\zeta(\max_{y\in\Gamma}d(y,\Lambda)+\varepsilon
,0)+\varepsilon,
\]
using that $\zeta$ is increasing in the first argument and decreasing in the
second argument.

(i) Fix $\varepsilon>0$ and let $\tilde{\varepsilon},\tau>0$ be given. Since
the compact set $P_{\varepsilon+2\tilde{\varepsilon}}$ is contained in the
interior of $P_{\varepsilon+3\tilde{\varepsilon}}$ one can choose a $C^{1}%
$-function $\theta:\mathbb{R}^{d}\rightarrow\lbrack0,1]$ with $\theta(x)=1$
for all $x\in P_{\varepsilon+2\tilde{\varepsilon}}$ and support contained in
$P_{\varepsilon+3\tilde{\varepsilon}}$ (cf. Abraham, Marsden and Ratiu
\cite[Prop.~5.5.8, p.~380]{AbrMR88}). We define $\tilde{f}:\mathbb{R}%
^{d}\times\mathbb{R}^{m}\rightarrow\mathbb{R}^{d}$ by $\tilde{f}%
(x,u):=\theta(x)f(x,u)$ (note that $\tilde{f}$ depends on $\tilde{\varepsilon
}$). Then $\tilde{f}$ is continuous and the derivative with respect to the
first argument is continuous in $(x,u)$. Consider the control system
\begin{equation}
\dot{x}(t)=\tilde{f}(x(t),u(t)),\ \ u(t)\in U. \label{eq_newcs}%
\end{equation}
The right hand side of this system is globally bounded and thus solutions
exist globally (see e.g.~Sontag \cite[Prop.~C.3.7]{Sont98}). We denote the
solution map associated with \eqref{eq_newcs} by $\psi$ and observe that%
\begin{align*}
&  \left(  \psi([0,\tau],x_{0},u)\subset P_{\varepsilon+2\tilde{\varepsilon}%
}\text{ or }\varphi([0,\tau],x_{0},u)\subset P_{\varepsilon+2\tilde
{\varepsilon}}\right) \\
&  \Rightarrow\ \psi(t,x_{0},u)=\varphi(t,x_{0},u)\mbox{\ \ for all\ }t\in
\lbrack0,\tau].
\end{align*}
A global Lipschitz constant for $\tilde{f}$ on $\mathbb{R}^{d}\times U$ with
respect to the first variable is given by
\begin{equation}
\tilde{L}:=\max\left\{  \left\Vert {}\right.  \tilde{f}_{x}(x,u)\left.
{}\right\Vert \left\vert (x,u)\in\mathbb{R}^{d}\times U\right.  \right\}  ,
\label{L_eps0}%
\end{equation}
which satisfies%
\[
\tilde{L}=\tilde{L}_{\varepsilon+3\tilde{\varepsilon}}:=\max\left\{
\left\Vert {}\right.  \tilde{f}_{x}(x,u)\left.  {}\right\Vert \left\vert
(x,u)\in P_{\varepsilon+3\tilde{\varepsilon}}\times U\right.  \right\}  .
\]
Using continuity of $\tilde{f}_{x}$ and $\zeta$ and compactness of
$P_{\varepsilon+3\tilde{\varepsilon}}\times U$ one finds%
\begin{equation}
\tilde{L}_{\varepsilon+3\tilde{\varepsilon}}\rightarrow L_{\varepsilon}\text{
for }\tilde{\varepsilon}\rightarrow0. \label{L_eps2}%
\end{equation}
Every $(y,u)$ in $R_{\varepsilon+\tilde{\varepsilon}}$ satisfies
$\varphi(t,y,u)\in P_{\varepsilon+\tilde{\varepsilon}}$ and hence
$\varphi(t,y,u)=\psi(t,y,u),t\geq0$.

Now let $\mathcal{S}^{+}=\{(y_{1},u_{1}),\ldots,(y_{n},u_{n})\}\subset
R_{\varepsilon+\tilde{\varepsilon}}$ be a subset with the property that for
every $x_{0}\in\Gamma$ there exists $(y_{i},u_{i})\in\mathcal{S}^{+}$ with%
\[
\max_{t\in\lbrack0,\tau]}d(\psi(t,x_{0},u_{i}),\psi(t,y_{i},u_{i}%
))<\tilde{\varepsilon}.
\]
Thus also $\psi(t,x_{0},u_{i})=\varphi(t,x_{0},u_{i})$, since $\varphi
(t,x_{0},u_{i})\in P_{\varepsilon+2\tilde{\varepsilon}},t\in\lbrack0,\tau]$.
By continuity and compactness, we may in fact assume that $\mathcal{S}^{+}$
has finite cardinality, and we take $\mathcal{S}^{+}$ with minimal cardinality
denoted by $r^{+}(\tau,\tilde{\varepsilon})$. We claim that for $2\tilde
{\varepsilon}<\varepsilon$,%
\begin{equation}
r_{\mathrm{ps}}(\tau,2\varepsilon,\zeta,\Gamma,\Lambda)\leq r_{\mathrm{ps}%
}(\tau,\varepsilon+2\tilde{\varepsilon},\zeta,\Gamma,\Lambda)\leq r^{+}%
(\tau,\tilde{\varepsilon}). \label{strong}%
\end{equation}
The first inequality follows by monotonicity in the second argument. The
second inequality follows, since for a minimal set $\mathcal{S}^{+}$ as above
and $x_{0}\in\Gamma$ there are $(y_{i},u_{i})\in\mathcal{S}^{+}$ such that for
all $t\in\lbrack0,\tau]$%
\begin{align*}
d(\psi(t,x_{0},u_{i}),\Lambda)  &  <d(\psi(t,x_{0},u_{i}),\psi(t,y_{i}%
,u_{i}))+d(\psi(t,y_{i},u_{i}),\Lambda)\\
&  <\tilde{\varepsilon}+\zeta(d(y_{i},\Lambda)+\varepsilon+\tilde{\varepsilon
},t)+\varepsilon+\tilde{\varepsilon}\\
&  \leq\zeta(\left\Vert y_{i}-x_{0}\right\Vert +d(x_{0},\Lambda)+\varepsilon
+\tilde{\varepsilon},t)+\varepsilon+2\tilde{\varepsilon}\\
&  <\zeta(d(x_{0},\Lambda)+\varepsilon+2\tilde{\varepsilon},t)+\varepsilon
+2\tilde{\varepsilon}.
\end{align*}
Since $\psi(t,x_{0},u_{i})=\varphi(t,x_{0},u_{i}),t\in\lbrack0,\tau]$, it
follows that $\mathcal{S}^{+}$ is practically $(\tau,\varepsilon
+2\tilde{\varepsilon},\zeta,\Gamma,\Lambda)$-spanning, and (\ref{strong}) is proved.

Next define the sets
\[
\Gamma_{i}:=\left\{  x_{0}\in\Gamma\ |\ \max_{t\in\lbrack0,\tau]}%
d(\psi(t,x_{0},u_{i}),\psi(t,y_{i},u_{i}))<\tilde{\varepsilon}\right\}
,\ \ i=1,\ldots,n=r^{+}(\tau,\tilde{\varepsilon}),
\]
By the definitions, $\Gamma=\bigcup_{i=1}^{n}\Gamma_{i}$. Let $x_{0}%
\in\mathbb{R}^{d}$ be a point with $\Vert x_{0}-y_{i}\Vert<e^{-\tilde{L}\tau
}\tilde{\varepsilon}$ for some $i\in\{1,\ldots,r^{+}(\tau,\tilde{\varepsilon
})\}$. By (\ref{L_eps0}) it follows that
\begin{equation}
\Vert\psi(t,x_{0},u_{i})-\psi(t,y_{i},u_{i})\Vert\leq\Vert x_{0}-y_{i}%
\Vert+\tilde{L}\int_{0}^{t}\Vert\psi(\sigma,x_{0},u_{i})-\psi(\sigma
,y_{i},u_{i})\Vert d\sigma\label{lipest}%
\end{equation}
for all $t\geq0$. By Gronwall's Lemma this implies for all $t\in\lbrack
0,\tau]$,%
\begin{equation}
\Vert\psi(t,x_{0},u_{i})-\psi(t,y_{i},u_{i})\Vert\leq\Vert x_{0}-y_{i}\Vert
e^{\tilde{L}t}<\tilde{\varepsilon}. \label{Gron}%
\end{equation}
It follows that $x_{0}\in\Gamma_{i}$ and thus $\Gamma$ contains the union of
the balls $\mathbf{B}(y_{i},e^{-\tilde{L}\tau}\tilde{\varepsilon})$.

Now assume that there exists a cover $\mathcal{V}$ of $\Gamma$ consisting of
balls $\mathbf{B}(x_{i},e^{-\tilde{L}\tau}\tilde{\varepsilon}),x_{i}\in\Gamma$
for $i=1,\ldots,N$, such that $N=\#\mathcal{V}<\#\mathcal{S}^{+}=r^{+}%
(\tau,\tilde{\varepsilon})$. By assumption (\ref{ass2}) we can assign to each
point $x_{i}$ a control function $v_{i}$ with $(x_{i},v_{i})\in R_{\varepsilon
+\tilde{\varepsilon}}$. Then, by the arguments above, the ball $\mathbf{B}%
(x_{i},e^{-\tilde{L}\tau}\tilde{\varepsilon})$ is contained in the set
\[
\left\{  x_{0}\in\mathbb{R}^{d}\ |\ \max_{t\in\lbrack0,\tau]}d(\psi
(t,x_{0},v_{i}),\psi(t,x_{i},v_{i}))<\tilde{\varepsilon}\right\}  .
\]
This contradicts the minimality of $\mathcal{S}^{+}$. Let $c(\delta,Z)$ is the
minimal cardinality of a cover of a bounded subset $Z\subset\mathbb{R}^{d}$ by
$\delta$-balls. We have shown that $r^{+}(\tau,\tilde{\varepsilon})\leq
c(\delta,\Gamma)$ with $\delta:=e^{-\tilde{L}\tau}\tilde{\varepsilon}$.

Recall that for a bounded subset $Z\subset\mathbb{R}^{d}$ the upper box or
fractal dimension satisfies%
\[
\dim_{F}(Z):=\overline{\lim_{\delta\searrow0}}\frac{\log c(\delta,Z)}%
{\log(1/\delta)}\leq d,
\]
cf. e.g. Boichenko, Leonov, and Reitmann \cite[Proposition 2.2.2 in Chapter
III]{05BoiLR}. Since%
\[
\tilde{L}\tau=\log(e^{\tilde{L}\tau}\tilde{\varepsilon}^{-1})+\log
\tilde{\varepsilon}=\log(e^{\tilde{L}\tau}\tilde{\varepsilon}^{-1})\left(
1+\frac{\log\tilde{\varepsilon}}{\log(e^{\tilde{L}\tau}\tilde{\varepsilon
}^{-1})}\right)  ,
\]
it follows that%
\begin{align}
&  \underset{\tau\rightarrow\infty}{\overline{\lim}}\frac{1}{\tau}\log
r^{+}(\tau,\tilde{\varepsilon})\leq\underset{\tau\rightarrow\infty}%
{\overline{\lim}}\frac{1}{\tau}\log c(e^{-\tilde{L}\tau}\tilde{\varepsilon
},\Gamma)=\tilde{L}\underset{\tau\rightarrow\infty}{\overline{\lim}}\frac{\log
c(e^{-\tilde{L}\tau}\tilde{\varepsilon},\Gamma)}{\tilde{L}\tau}\nonumber\\
&  =\tilde{L}\underset{\tau\rightarrow\infty}{\overline{\lim}}\frac{\log
c(e^{-\tilde{L}\tau}\tilde{\varepsilon},\Gamma)}{\log(e^{\tilde{L}\tau}%
\tilde{\varepsilon}^{-1})(1+\frac{\log\tilde{\varepsilon}}{\log(e^{\tilde
{L}\tau}\tilde{\varepsilon}^{-1})})}=\tilde{L}\underset{\tau\rightarrow\infty
}{\overline{\lim}}\frac{\log c(e^{-\tilde{L}\tau}\tilde{\varepsilon},\Gamma
)}{\log(e^{\tilde{L}\tau}\tilde{\varepsilon}^{-1})}\label{h_strong}\\
&  =\tilde{L}\dim_{F}(\Gamma).\nonumber
\end{align}
As $\tilde{\varepsilon}$ tends to zero, the Lipschitz constants $\tilde
{L}=\tilde{L}_{\varepsilon+3\tilde{\varepsilon}}$ tend to $L_{\varepsilon}$ by
(\ref{L_eps2}). Taking into account also (\ref{strong}), this implies%
\[
h_{\mathrm{ps}}(2\varepsilon,\zeta,\Gamma,\Lambda)\leq\left(  \lim
_{\tilde{\varepsilon}\rightarrow0}\hat{L}_{\varepsilon+3\tilde{\varepsilon}%
}\right)  \dim_{F}(\Gamma)=L_{\varepsilon}\dim_{F}(\Gamma)\leq L_{\varepsilon
}d,
\]
which proves assertion (i).

(ii) If the assumptions in (i) hold for all $\varepsilon>0$, then the
Lipschitz constants $L_{\varepsilon}$ converge for $\varepsilon\rightarrow0$
to $L_{0}$ and the assertion follows from
\[
h_{\mathrm{ps}}(\zeta,\Gamma,\Lambda)=\lim_{\varepsilon\rightarrow
0}h_{\mathrm{ps}}(\varepsilon,\zeta,\Gamma,\Lambda)\leq\lim_{\varepsilon
\rightarrow0}L_{\varepsilon}d=L_{0}d.
\]

\end{proof}

The following theorem gives a similar estimate for the stabilization entropy
with exponential $\mathcal{KL}$-function. For compact control range $U$ and
$\varepsilon\geq0,M\geq1$, define the compact set%
\[
P_{\varepsilon}^{\mathrm{s}}:=\left\{  x\in\mathbb{R}^{d}\left\vert
d(x,\Lambda)\leq M(\max_{y\in\Gamma}d(y,\Lambda)+\varepsilon)\right.
\right\}
\]
and the constant $L_{0}^{\mathrm{s}}:=\max_{(x,u)\in P_{0}^{\mathrm{s}}\times
U}\left\Vert f_{x}(x,u)\right\Vert $.

\begin{theorem}
\label{Theorem_pract}Consider an exponential $\mathcal{KL}$-function
$\zeta(r,s)=e^{-\alpha s}Mr,r,s\geq0$, with constants $\alpha>0,M\geq1$, and
suppose that the assumptions of Theorem \ref{Theorem_upper} are satisfied for
control system (\ref{2.1}). Assume that for every $\varepsilon>0$ and for
every $x_{0}\in\Gamma$ there is a control $u\in\mathcal{U}$ with%
\begin{equation}
d(\varphi(t,x_{0},u),\Lambda)<e^{-\alpha t}M(d(x_{0},\Lambda)+\varepsilon
)\text{ for all }t\geq0. \label{ass3}%
\end{equation}
Then the \textit{stabilization entropy} satisfies $h_{\mathrm{s}}(\zeta
,\Gamma,\Lambda)\leq(L_{0}^{\mathrm{s}}+\alpha)d$.
\end{theorem}

\begin{proof}
This proof follows similar steps as the proof of Theorem \ref{Theorem_upper}
but it is somewhat simpler. Define for $\varepsilon\geq0$%
\[
R_{\varepsilon}^{\mathrm{s}}:=\{(x_{0},u)\in\Gamma\times\mathcal{U}\left\vert
d(\varphi(t,x_{0},u),\Lambda)<e^{-\alpha t}M(d(x_{0},\Lambda)+\varepsilon
)\text{ for all }t\geq0\right.  \}.
\]
Fix $\varepsilon>0$ and choose a $C^{1}$-function $\theta:\mathbb{R}%
^{d}\rightarrow\lbrack0,1]$ with $\theta(x)=1$ for all $x\in P_{3\varepsilon
}^{\mathrm{s}}$ and support contained in $P_{4\varepsilon}^{\mathrm{s}}$. We
define $\tilde{f}:\mathbb{R}^{d}\times\mathbb{R}^{m}\rightarrow\mathbb{R}^{d}$
by $\tilde{f}(x,u):=\theta(x)f(x,u)$ and consider the control system
\[
\dot{x}(t)=\tilde{f}(x(t),u(t)),\ \ u(t)\in U.
\]
The solution map $\psi$ associated with this system satisfies for $\tau>0$,%
\begin{align*}
&  \left(  \psi([0,\tau],x_{0},u)\subset P_{3\varepsilon}^{\mathrm{s}}\text{
or }\varphi([0,\tau],x_{0},u)\subset P_{3\varepsilon}^{\mathrm{s}}\right) \\
&  \Rightarrow\ \psi(t,x_{0},u)=\varphi(t,x_{0},u)\mbox{\ \ for all\ }t\in
\lbrack0,\tau].
\end{align*}
Now let $\mathcal{S}^{\ast}=\{(y_{1},u_{1}),\ldots,(y_{n},u_{n})\}\subset
R_{\varepsilon}^{\mathrm{s}},n=r^{\ast}(\tau,\varepsilon)$, be a minimal
subset with the property that for every $x_{0}\in\Gamma$ there exists
$(y_{i},u_{i})\in\mathcal{S}^{\ast}$ with%
\[
\max_{t\in\lbrack0,\tau]}d(\psi(t,x_{0},u_{i}),\psi(t,y_{i},u_{i}%
))<\varepsilon e^{-\alpha\tau}\text{ for }t\in\lbrack0,\tau].
\]
Define%
\[
\Gamma_{i}:=\left\{  x_{0}\in\Gamma\ |\ \max_{t\in\lbrack0,\tau]}%
d(\psi(t,x_{0},u_{i}),\psi(t,y_{i},u_{i}))<\varepsilon e^{-\alpha\tau
}\right\}  ,\ \ i=1,\ldots,n=r^{\ast}(\tau,\tilde{\varepsilon}),
\]%
\[
L_{4\varepsilon}^{\mathrm{s}}:=\max\left\{  \left\Vert {}\right.  \tilde
{f}_{x}(x,u)\left.  {}\right\Vert \left\vert (x,u)\in P_{4\varepsilon}\times
U\right.  \right\}  =\max\left\{  \left\Vert {}\right.  \tilde{f}%
_{x}(x,u)\left.  {}\right\Vert \left\vert (x,u)\in\mathbb{R}^{d}\times
U\right.  \right\}  .
\]
Consider $x_{0}\in\mathbb{R}^{d}$ with $\Vert x_{0}-y_{i}\Vert<e^{-\left(
L_{4\varepsilon}^{s}+\alpha\right)  \tau}\varepsilon$ for some $i\in
\{1,\ldots,n\}$. Then Gronwall's Lemma implies instead of (\ref{Gron}) for
$t\in\lbrack0,\tau]$%
\[
\Vert\psi(t,x_{0},u_{i})-\psi(t,y_{i},u_{i})\Vert\leq\Vert x_{0}-y_{i}\Vert
e^{L_{4\varepsilon}^{\mathrm{s}}\tau}<e^{-\left(  L_{4\varepsilon}%
^{\mathrm{s}}+\alpha\right)  \tau}\varepsilon e^{L_{4\varepsilon}^{\mathrm{s}%
}\tau}=\varepsilon e^{-\alpha\tau}.
\]
It follows that $x_{0}\in\Gamma_{i}$ and thus $\Gamma$ contains the union of
the balls $\mathbf{B}(y_{i},e^{-\left(  L_{4\varepsilon}^{\mathrm{s}}%
+\alpha\right)  \tau}\varepsilon)$. Instead of (\ref{strong}) we obtain
$r_{\mathrm{s}}(\tau,3\varepsilon,\zeta,\Gamma,\Lambda)\leq r^{\ast}%
(\tau,\varepsilon)$, since for a minimal set $\mathcal{S}^{\ast}$ as above and
$x_{0}\in\Gamma$ there are $(y_{i},u_{i})\in\mathcal{S}^{\ast}$ such that for
all $t\in\lbrack0,\tau]$%
\begin{align*}
d(\psi(t,x_{0},u_{i}),\Lambda)  &  <d(\psi(t,x_{0},u_{i}),\psi(t,y_{i}%
,u_{i}))+d(\psi(t,y_{i},u_{i}),\Lambda)\\
&  <e^{-\alpha t}\varepsilon+e^{-\alpha t}M(d(y_{i},\Lambda)+\varepsilon)\\
&  \leq e^{-\alpha t}M(\left\Vert y_{i}-x_{0}\right\Vert +d(x_{0}%
,\Lambda)+2\varepsilon)\\
&  \leq e^{-\alpha t}M(d(x_{0},\Lambda)+3\varepsilon).
\end{align*}
Furthermore $\psi(t,x_{0},u_{i})\in P_{3\varepsilon}^{\mathrm{s}}$, hence
$\psi(t,x_{0},u_{i})=\varphi(t,x_{0},u_{i}),t\in\lbrack0,\tau]$, and the set
of controls $\{u_{1},\ldots,u_{n}\},\allowbreak n=r^{\ast}(\tau,\varepsilon)$,
is $\left(  \tau,3\varepsilon,\Gamma,\Lambda\right)  $-spanning. One finds
that%
\[
r_{\mathrm{s}}(\tau,3\varepsilon,\zeta,\Gamma,\Lambda)\leq r^{\ast}%
(\tau,\varepsilon)\leq c(\delta,\Gamma)\text{ with }\delta:=e^{-\left(
L_{4\varepsilon}^{\mathrm{s}}+\alpha\right)  \tau}\varepsilon,
\]
and
\[
\underset{\tau\rightarrow\infty}{\overline{\lim}}\frac{1}{\tau}\log r^{\ast
}(\tau,\varepsilon)\leq(L_{4\varepsilon}^{\mathrm{s}}+\alpha)\dim_{F}%
(\Gamma)\leq(L_{4\varepsilon}^{\mathrm{s}}+\alpha)d.
\]
For $\varepsilon\rightarrow0$ the Lipschitz constants $L_{4\varepsilon
}^{\mathrm{s}}$ converge to $L_{0}^{\mathrm{s}}$ and the assertion follows.
\end{proof}

\begin{remark}
Theorem \ref{Theorem_pract} generalizes and improves Colonius \cite[Theorem
3.3]{Colo12b}, where an upper bound for stabilization entropy about an
equilibrium is given using a global Lipschitz constant $L$.
\end{remark}

Next we prove lower bounds for the $\varepsilon$-practical stabilization
entropy based on a volume growth argument. For general $\mathcal{KL}%
$-functions the lower bound is given by the divergence $\mathrm{div}%
_{x}f(x,u)=\mathrm{tr}f_{x}(x,u)$ of $f$ with respect to $x$, while for
exponential $\mathcal{KL}$-functions a stronger result involving also the
exponential bound holds.

\begin{theorem}
\label{Theorem_lower}Consider for control system (\ref{2.1}) compact sets
$\Gamma,\Lambda\subset\mathbb{R}^{d}$, where $\Gamma$ has positive Lebesgue
measure and let $\zeta$ be a $\mathcal{KL}$-function. Suppose that $f$ is
continuous and $f$ is differentiable with respect to $x$ and the partial
derivative $f_{x}(x,u)$ is continuous in $(x,u)$ with $\inf_{(x,u)\in A\times
U}f_{x}(x,u)>-\infty$ for bounded sets $A\subset\mathbb{R}^{d}$.

(i) For $\varepsilon>0$ the $\varepsilon$-practical stabilization entropy and
the practical stabilization entropy satisfy%
\begin{align*}
\infty &  \geq h_{\mathrm{ps}}(\varepsilon,\zeta,\Gamma,\Lambda)\geq
\min_{(x,u)\in\overline{N(\Lambda;\varepsilon)}\times U}\mathrm{div}%
_{x}f(x,u),\\
\infty &  \geq h_{\mathrm{ps}}(\zeta,\Gamma,\Lambda)\geq\min_{(x,u)\in
\Lambda\times U}\mathrm{div}_{x}f(x,u).
\end{align*}

(ii) Let $\Lambda=\{0\}$ and $\zeta(r,s):=e^{-\alpha s}Mr,r,s\geq0$, with
constants $\alpha>0,M\geq1$. Then the $\varepsilon$-practical stabilization
entropy and the practical stabilization entropy satisfy%
\begin{align*}
\infty &  \geq h_{\mathrm{ps}}(\varepsilon,\zeta,\Gamma,\{0\})\geq\alpha
d+\min_{\left\Vert x\right\Vert \leq\varepsilon,u\in U}\mathrm{div}%
_{x}f(x,u),\\
\infty &  \geq h_{\mathrm{ps}}(\zeta,\Gamma,\{0\})\geq\alpha d+\min_{u\in
U}\mathrm{div}_{x}f(0,u).
\end{align*}

\end{theorem}

\begin{proof}
(i) If $h_{\mathrm{ps}}(\varepsilon,\zeta,\Gamma,\Lambda)=\infty$, the
inequalities in (i) are trivially satisfied. Hence we may assume that for
$\tau>0$ there is a finite practically $(\tau,\varepsilon,\zeta,\Gamma
,\Lambda)$-spanning set $\mathcal{S}=\{u_{1},\ldots,u_{n}\}$ of controls and
we pick $\mathcal{S}$ with minimal cardinality, hence $n=r_{\mathrm{ps}}%
(\tau,\varepsilon,\zeta,\Gamma,\Lambda)$. Define for $i=1,\ldots,n$%
\[
\Gamma_{i}:=\left\{  x_{0}\in\Gamma\left\vert d(\varphi(t,x_{0},u_{i}%
),\Lambda)<\zeta(d(x_{0},\Lambda)+\varepsilon,t)+\varepsilon\text{ for all
}t\in\lbrack0,\tau]\right.  \right\}  .
\]
Denote $\kappa:=\max_{x\in\Gamma}d(x,\Lambda)$ and $\delta(t):=\zeta
(\kappa+\varepsilon,t),t\in\lbrack0,\tau]$. Then for $i=1,\ldots,n$
\begin{equation}
\varphi(t,\Gamma_{i},u_{i})\subset N(\Lambda,\zeta(\kappa+\varepsilon
,t)+\varepsilon)=N(\Lambda;\delta(t)+\varepsilon),t\in\lbrack0,\tau],
\label{rev1}%
\end{equation}
implying for the Lebesgue measures%
\begin{equation}
\lambda(\varphi(t,\Gamma_{i},u_{i}))\leq\lambda(N(\Lambda;\delta
(t)+\varepsilon)). \label{4.3v}%
\end{equation}
On the other hand, by the transformation theorem for diffeomorphisms and
Liouville's trace formula (cf. Teschl \cite[Lemma 3.11]{Teschl}) we get for
$i=1,\ldots,n$
\begin{align}
\lambda(\varphi(\tau,\Gamma_{i},u_{i}))  &  =\int_{\Gamma_{i}}\left\vert
\det\frac{\partial\varphi}{\partial x_{0}}(\tau,x_{0},u_{i})\right\vert
dx_{0}\geq\lambda(\Gamma_{i})\cdot\inf_{(x_{0},u)}\left\vert \det
\frac{\partial\varphi}{\partial x_{0}}(\tau,x_{0},u)\right\vert \nonumber\\
&  =\lambda(\Gamma_{i})\cdot\inf_{(x_{0},u)}\exp\left(  \int_{0}^{\tau
}\mathrm{div}_{x}f(\varphi(s,x_{0},u),u(s))ds\right)  . \label{4.4v}%
\end{align}
Here, and in the rest of this proof, $\inf_{(x_{0},u)}$ denotes the infimum
over all $(x_{0},u)\in\Gamma\times\mathcal{U}$ with $\varphi(t,x_{0},u)\subset
N(\Lambda;\delta(t)+\varepsilon)$ for all $t\in\lbrack0,\tau]$. Fix $\tau
_{0}\in\lbrack0,\tau]$. Then $(x_{0},u)$ as above satisfies the estimate
\begin{align}
&  \int_{0}^{\tau}\mathrm{div}_{x}f(\varphi(s,x_{0},u),u(s))ds\nonumber\\
&  \geq\int_{0}^{\tau_{0}}\mathrm{div}_{x}f(\varphi(s,x_{0},u),u(s))ds+(\tau
-\tau_{0})\min_{(y,v)}\mathrm{div}_{x}f(y,v), \label{4.5vA}%
\end{align}
where the minimum is taken over all $(y,v)\in\mathbb{R}^{d}\times U$ with
$d(y,\Lambda)\leq\zeta(\kappa+\varepsilon,\tau_{0})+\varepsilon=\delta
(\tau_{0})+\varepsilon$. This holds since the function $\zeta$ is decreasing
in the second argument, and for all $s\in\lbrack\tau_{0},\tau]$%
\[
d(\varphi(s,x_{0},u),\Lambda)<\zeta(\max_{x_{0}\in\Gamma}d(x_{0}%
,\Lambda)+\varepsilon,s)+\varepsilon\leq\zeta(\kappa+\varepsilon,\tau
_{0})+\varepsilon.
\]
We may assume that $\lambda(\Gamma_{1})=\max_{i=1,\ldots,n}\lambda(\Gamma
_{i})$. Inequalities (\ref{4.4v}) and (\ref{4.3v}) imply
\begin{align*}
0  &  <\lambda(\Gamma)\leq\sum_{i=1}^{n}\lambda(\Gamma_{1})\leq n\cdot
\lambda(\Gamma_{1})\leq n\cdot\frac{\lambda(\varphi(\tau,\Gamma_{1},u_{1}%
))}{\inf_{(x_{0},u)}\exp\left(  \int_{0}^{\tau}\mathrm{div}_{x}f(\varphi
(s,x_{0},u),u(s))ds\right)  }\\
&  \leq n\cdot\frac{\lambda(N(\Lambda;\delta(\tau)+\varepsilon))}{\inf
_{(x_{0},u)}\exp\left(  \int_{0}^{\tau}\mathrm{div}_{x}f(\varphi
(s,x_{0},u),u(s))ds\right)  },
\end{align*}
hence%
\[
n=r_{\mathrm{ps}}(\tau,\varepsilon,\zeta,\Gamma,\Lambda)\geq\frac
{\lambda(\Gamma)}{\lambda(N(\Lambda;\delta(\tau)+\varepsilon))}\inf
_{(x_{0},u)}\exp\left(  \int_{0}^{\tau}\mathrm{div}_{x}f(\varphi
(s,x_{0},u),u(s))ds\right)  .
\]
Using (\ref{4.5vA}) and taking the logarithm on both sides one finds%
\begin{align*}
&  \log r_{\mathrm{ps}}(\tau,\varepsilon,\zeta,\Gamma,\Lambda)\geq\log
\lambda(\Gamma)-\log\lambda(N(\Lambda;\delta(\tau)+\varepsilon))+\\
&  \qquad+\inf_{(x_{0},u)}\int_{0}^{\tau_{0}}\mathrm{div}_{x}f(\varphi
(s,x_{0},u),u(s))ds+(\tau-\tau_{0})\min_{(y,v)}\mathrm{div}_{x}f(y,v).
\end{align*}
This yields the inequality
\begin{align}
&  \underset{\tau\rightarrow\infty}{\overline{\lim}}\frac{1}{\tau}\log
r_{\mathrm{ps}}(\tau,\varepsilon,\zeta,\Gamma,\Lambda)\label{rev3}\\
&  \geq\underset{\tau\rightarrow\infty}{\overline{\lim}}\left[  -\frac{1}%
{\tau}\log\lambda(N(\Lambda;\delta(\tau)+\varepsilon))+\frac{\tau-\tau_{0}%
}{\tau}\min_{(y,v)}\mathrm{div}_{x}f(y,v)\right]  .\nonumber
\end{align}
Since $\delta(\tau)\leq\delta(0)$ and $\lambda(N(\Lambda;\varepsilon))>0$, we
find%
\[
\underset{\tau\rightarrow\infty}{\overline{\lim}}-\frac{1}{\tau}\log
\lambda(N(\Lambda;\delta(\tau)+\varepsilon))=0.
\]
It follows that%
\[
\underset{\tau\rightarrow\infty}{\overline{\lim}}\frac{1}{\tau}\log
r_{\mathrm{ps}}(\tau,\varepsilon,\zeta,\Gamma,\Lambda)\geq\min_{(y,v)}%
\mathrm{div}_{x}f(y,v).
\]
Recall that the minimum is on the set $\left\{  (y,v)\in\mathbb{R}^{d}\times
U\left\vert d(y,\Lambda)\leq\delta(\tau_{0})+\varepsilon\right.  \right\}  $.
In the Hausdorff metric, these compact sets converge to $\overline
{N(\Lambda;\varepsilon)}\times U$ for $\tau_{0}\rightarrow\infty$. This proves
the first assertion in (i), the second follows by taking the limit for
$\varepsilon\rightarrow0$.\medskip

(ii) For $\varepsilon>0$ inequality (\ref{rev3}) holds. If we employ the
maximum-norm in $\mathbb{R}^{d}$, we obtain for the Lebesgue measure%
\[
\lambda(N(\{0\},\delta(\tau)+\varepsilon))=\lambda(\mathbf{B}(0,\delta
(\tau)+\varepsilon))\leq(2\delta(\tau)+2\varepsilon)^{d},
\]
and, by the choice of $\zeta$,%
\[
\lim_{\tau\rightarrow\infty}\frac{1}{\tau}\log(\delta(\tau)+\varepsilon
)=\lim_{\tau\rightarrow\infty}\frac{1}{\tau}\log\left[  e^{-\alpha\tau
}M(\kappa+\varepsilon)+\varepsilon\right]  =-\alpha.
\]
Hence (\ref{rev3}) implies
\begin{align*}
\underset{\tau\rightarrow\infty}{\overline{\lim}}\frac{1}{\tau}\log
r_{\mathrm{ps}}(\tau,\varepsilon,\zeta,\Gamma,\Lambda)  &  \geq\alpha
d+\underset{\tau\rightarrow\infty}{\overline{\lim}}\left(  \frac{\tau-\tau
_{0}}{\tau}\min_{(y,v)}\mathrm{div}_{x}f(y,v)\right) \\
&  =\alpha d+\min_{(y,v)}\mathrm{div}_{x}f(y,v)
\end{align*}
and the inequalities in assertion (ii) follow as in (i).
\end{proof}

\begin{remark}
Note that the lower bounds provided by Theorem \ref{Theorem_lower} may be
negative. In the linear case, this can be improved, cf. Theorem
\ref{Theorem6.2}.
\end{remark}

\begin{remark}
Theorem \ref{Theorem_lower}(ii) improves Colonius \cite[Theorem 3.2]{Colo12b},
where a similar lower bound for the stabilization entropy (which may be
greater than the practical stabilization entropy) is proved.
\end{remark}

\section{Relations to feedbacks\label{Section4}}

We will prove an upper bound of the $\varepsilon$-stabilization entropy under
the assumption that a feedback exists such that the system satisfies an
appropriate stability property. This is illustrated in the linear case.

Consider for system (\ref{2.1}) a Lipschitz continuous feedback $k:\mathbb{R}%
^{d}\rightarrow U$ such that the solutions $\psi(t,x_{0};k(\cdot)),t\geq0$, of%
\begin{equation}
x(0)=x_{0},\dot{x}(t)=f(x(t),k(x(t))), \label{F_system}%
\end{equation}
are well defined and depend continuously on the initial value. Fix
$\varepsilon>0$ and let $\zeta$ be a $\mathcal{KL}$-function. Define the
$\varepsilon$-entropy of $k(\cdot)$ in the following way. Let $\Gamma
\subset\mathbb{R}^{d}$ be compact and define for every $x_{0}\in\Gamma$ a
control by%
\begin{equation}
u_{x_{0}}(t)=k(\psi(t,x_{0};k(\cdot))),t\geq0. \label{4.2}%
\end{equation}
For $\tau>0$ a set $E=\{y_{1},\ldots,y_{n}\}\subset\Gamma$ is $(\tau
,\varepsilon,\zeta,\Gamma)$-spanning for $k(\cdot)$ if for all $x_{0}\in
\Gamma$ there is $j\in\{1,\ldots,n\}$ with%
\[
\left\Vert x_{0}-y_{j}\right\Vert <\varepsilon\text{ and }\left\Vert
\varphi(t,x_{0},u_{y_{j}})-\varphi(t,y_{j},u_{y_{j}})\right\Vert \leq
\zeta(\left\Vert x_{0}-y_{j}\right\Vert +\varepsilon,t)\text{ for }t\in
\lbrack0,\tau].
\]

\begin{definition}
\label{Definition_F_entropy}For system (\ref{2.1}), a set $\Gamma
\subset\mathbb{R}^{d}$ of initial states, and a $\mathcal{KL}$-function
$\zeta$ the $\varepsilon$-entropy of the feedback $k(\cdot)$ is%
\[
h_{\mathrm{fb}}(\varepsilon,\zeta,k(\cdot),\Gamma)=\underset{\tau
\rightarrow\infty}{\overline{\lim}}\frac{1}{\tau}\log\min\{\#E\left\vert
E\text{ is }(\tau,\varepsilon,\zeta,\Gamma)\text{-spanning for }%
k(\cdot)\right.  \}.
\]
If the feedback $k(\cdot)$ is independent of $\varepsilon$, we define the
entropy of the feedback $k(\cdot)$ as%
\[
h_{\mathrm{fb}}(\zeta,k(\cdot),\Gamma)=\lim_{\varepsilon\rightarrow
0}h_{\mathrm{fb}}(\varepsilon,\zeta,k(\cdot),\Gamma).
\]

\end{definition}

These notions of entropy are based on the concept, that only in the beginning,
at time $t=0$, an estimate of the initial point is used. The control is not
corrected at any later time.

The following proposition shows that the $\varepsilon$-stabilization entropy
can be bounded above by the $\varepsilon$-entropy of feedbacks for which the
system satisfies an $\varepsilon$-stability property.

\begin{proposition}
\label{Proposition_upper}Let $\Gamma,\Lambda\subset\mathbb{R}^{d}$ be compact
and $\varepsilon>0$. Suppose that there is\ a feedback $k_{\varepsilon}%
(\cdot)$ such that for every $x_{0}\in\Gamma$ the solution $\psi
(t,x_{0};k_{\varepsilon}(\cdot)),t\geq0$, of the feedback system
(\ref{F_system}) satisfies%
\begin{equation}
d(\psi(t,x_{0};k_{\varepsilon}(\cdot)),\Lambda)\leq\zeta(d(x_{0}%
,\Lambda)+\varepsilon,t)\text{ for }t\geq0. \label{eps0}%
\end{equation}
Then the $\varepsilon$-stabilization entropy is bounded above by the
$\varepsilon$-entropy of the feedback $k_{\varepsilon}(\cdot)$,%
\[
h_{\mathrm{ps}}(2\varepsilon,2\zeta,\Gamma,\Lambda)\leq h_{\mathrm{s}%
}(2\varepsilon,2\zeta,\Gamma,\Lambda)\leq h_{\mathrm{fb}}(\varepsilon
,\zeta,k_{\varepsilon}(\cdot),\Gamma).
\]

\end{proposition}

\begin{proof}
Note that the first inequality trivially holds. Fix $\tau>0$ and consider a
minimal $(\tau,\varepsilon,\zeta,\Gamma)$-spanning set $E=\{y_{1},\ldots
,y_{n}\}$ for $k_{\varepsilon}(\cdot)$. Using (\ref{4.2}) we associate to
every $y_{i}$ a control function $u_{y_{i}}$ and claim that the set of control
functions defined by%
\[
\mathcal{S}(E)=\{u_{y_{1}},\ldots,u_{y_{n}}\}
\]
is $(\tau,2\varepsilon,2\zeta,\Gamma,\Lambda)$-spanning. By assumption
(\ref{eps0}) for the feedback system we know%
\[
d(\varphi(t,y_{j},u_{y_{j}}),\Lambda)=d(\psi(t,y_{j};k_{\varepsilon}%
(\cdot)),\Lambda)\leq\zeta\left(  d(y_{j},\Lambda)+\varepsilon,t\right)
\text{ for all }t\geq0.
\]
By the spanning property of $E$, for all $x_{0}\in\Gamma$ there is $j$ such
that for all $t\in\lbrack0,\tau]$,
\begin{align}
d\left(  \varphi(t,x_{0},u_{y_{j}}),\Lambda\right)   &  \leq\left\Vert
\varphi(t,x_{0},u_{y_{j}})-\varphi(t,y_{j},u_{y_{j}})\right\Vert
+d(\varphi(t,y_{j},u_{y_{j}}),\Lambda)\nonumber\\
&  \leq\zeta(\left\Vert x_{0}-y_{j}\right\Vert +\varepsilon,t)+d(\psi
(t,y_{j};k_{\varepsilon}(\cdot)),\Lambda)\nonumber\\
&  \leq\zeta(2\varepsilon,t)+\zeta\left(  d(y_{j},\Lambda)+\varepsilon
,t\right) \label{eps1}\\
&  \leq\zeta(2\varepsilon,t)+\zeta\left(  \left\Vert y_{j}-x_{0}\right\Vert
+d(x_{0},\Lambda)+\varepsilon,t\right) \nonumber\\
&  \leq2\zeta(d(x_{0},\Lambda)+2\varepsilon,t).\nonumber
\end{align}
This shows the claim for $\mathcal{S}(E)$ and it follows that%
\begin{align*}
&  \min\{\#\mathcal{S}\left\vert \mathcal{S}\text{ is }(\tau,2\varepsilon
,2\zeta,\Gamma,\Lambda))\text{-spanning}\right.  \}\\
&  \leq\min\{\#E\left\vert E\text{ is }(\tau,\varepsilon,\zeta,\Gamma
)\text{-spanning for }k_{\varepsilon}(\cdot)\right.  \}.
\end{align*}
Taking logarithms and the limit for $\tau\rightarrow\infty$, one obtains the
assertion,%
\begin{align*}
&  h_{\mathrm{s}}(2\varepsilon,2\zeta,\Gamma,\Lambda)=\underset{\tau
\rightarrow\infty}{\overline{\lim}}\frac{1}{\tau}\log\min\{\#\mathcal{S}%
\left\vert \mathcal{S}\text{ is }(\tau,2\varepsilon,2\zeta,\Gamma
,\Lambda)\text{-spanning}\right.  \}\\
&  \leq\underset{\tau\rightarrow\infty}{\overline{\lim}}\frac{1}{\tau}\log
\min\{\#E\left\vert E\text{ is }(\tau,\varepsilon,\Gamma)\text{-spanning for
}k_{\varepsilon}(\cdot)\right.  \}=h_{\mathrm{fb}}(\varepsilon,\zeta
,k_{\varepsilon}(\cdot),\Gamma).
\end{align*}

\end{proof}

\begin{remark}
If one replaces (\ref{eps0}) by the weaker condition%
\[
d(\psi(t,x_{0};k_{\varepsilon}(\cdot)),\Lambda)\leq\zeta(d(x_{0}%
,\Lambda)+\varepsilon,t)+\varepsilon\text{ for }t\geq0,
\]
one can prove analogously a bound for the $\varepsilon$-practical
stabilization entropy,%
\[
h_{\mathrm{ps}}(2\varepsilon,2\zeta,\Gamma,\Lambda)\leq h_{\mathrm{fb}%
}(\varepsilon,\zeta,k_{\varepsilon}(\cdot),\Gamma).
\]

\end{remark}

Next we illustrate Proposition \ref{Proposition_upper} by considering linear
systems of the form%
\begin{equation}
\dot{x}(t)=Ax(t)+Bu(t),\ u\in\mathcal{U}, \label{linear}%
\end{equation}
with matrices $A\in\mathbb{R}^{d\times d}$, $B\in\mathbb{R}^{d\times m}$,
control range $U=\mathbb{R}^{m}$, and $\Lambda=\{0\}$. For a linear feedback
$K$ the feedback system has the form%
\[
\dot{x}(t)=(A+BK)x(t),~u\in\mathcal{U},
\]
with solutions $\psi(t,x_{0};K)=e^{(A+BK)t}x_{0}$. Suppose that $K$ is
stabilizing such that all eigenvalues $\lambda_{j}$ of $A+BK$ satisfy
$\operatorname{Re}\lambda_{j}<-\alpha$ for some $\alpha>0$. Hence the
solutions of the feedback system satisfy for every initial value $x_{0}%
\in\mathbb{R}^{d}$ and some constant $M\geq1$%
\begin{equation}
\left\Vert \psi(t,x_{0};K)\right\Vert \leq e^{-\alpha t}M\left\Vert
x_{0}\right\Vert =\zeta(\left\Vert x_{0}\right\Vert ,t), \label{eps3}%
\end{equation}
where $\zeta(r,s):=e^{-\alpha s}Mr,r,s\geq0$. Thus assumption (\ref{eps0}) in
Proposition \ref{Proposition_upper} holds, hence $h_{\mathrm{s}}%
(2\varepsilon,2\zeta,\Gamma,\{0\})\leq h_{\mathrm{fb}}(\varepsilon
,\zeta,K,\Gamma)$. For a compact subset $\Gamma\subset\mathbb{R}^{d}$ the
$\varepsilon$-entropy of $K$ is determined by the following. Let for $y_{0}%
\in\mathbb{R}^{d}$%
\[
u_{y_{0}}(t)=K\psi(t,y_{0};K)=Ke^{(A+BK)t}y_{0},t\geq0.
\]
Note that for $x_{0},y_{0}\in\mathbb{R}^{d}$%
\begin{align*}
&  \varphi(t,x_{0},u_{y_{0}})-\varphi(t,y_{0},u_{y_{0}})\\
&  =e^{At}x_{0}+\int_{0}^{t}e^{A(t-s)}Bu_{y_{0}}(s)ds-e^{At}y_{0}-\int_{0}%
^{t}e^{A(t-s)}Bu_{y_{0}}(s)ds\\
&  =e^{At}(x_{0}-y_{0}).
\end{align*}
For $\varepsilon,\tau>0$ a set $E=\{y_{1},\ldots,y_{n}\}$ is $(\tau
,\varepsilon,\zeta,\Gamma)$-spanning for the feedback $K$, if for all
$x_{0}\in\Gamma$ there exists $j$ such that $\left\Vert x_{0}-y_{j}\right\Vert
<\varepsilon$ and for $t\in\lbrack0,\tau]$%
\begin{equation}
\left\Vert \varphi(t,x_{0},u_{y_{j}})-\varphi(t,y_{j},u_{y_{j}})\right\Vert
=\left\Vert e^{At}(x_{0}-y_{j})\right\Vert \leq e^{-\alpha t}M(\left\Vert
x_{0}-y_{j}\right\Vert +\varepsilon). \label{4.eps}%
\end{equation}
A classical result shows that the topological entropy $h_{top}(\Phi_{t})$ of a
linear flow $\Phi_{t}=e^{At},t\in\mathbb{R}$, is given by $h_{\mathrm{top}%
}(\Phi_{t})=\sum_{i:\ \operatorname{Re}\lambda_{i}>0}\operatorname{Re}%
\lambda_{i}$, where $\lambda_{1},\ldots,\lambda_{n}$ denote the eigenvalues of
$A$, cf. Walters \cite[Theorem 8.14]{Walt82}. It follows that the topological
entropy of the flow $e^{(A+\alpha I)t},t\in\mathbb{R}$, is $\sum
_{i:\ \operatorname{Re}\lambda_{i}>-\alpha}\operatorname{Re}\lambda_{i}$. By
the definition of topological entropy of\ flows, for any compact set
$\Gamma\subset\mathbb{R}^{d}$ a set $F=\{z_{1},\ldots z_{k}\}$ is
$(\tau,\varepsilon,\Gamma)$-spanning if for every $x_{0}\in\Gamma$ there is
$z_{j}$ such that for $t\in\lbrack0,\tau]$
\[
\left\Vert e^{(A+\alpha I)t}(x_{0}-z_{j})\right\Vert <\varepsilon\text{, hence
}\left\Vert e^{At}(x_{0}-z_{j})\right\Vert <e^{-\alpha t}\varepsilon.
\]
This shows that $F$ is also $(\tau,\varepsilon,\zeta,\Gamma)$-spanning for the
feedback $K$, cf. (\ref{4.eps}). Using Proposition \ref{Proposition_upper} one
finds for $\tau\rightarrow\infty$ that the $\varepsilon$-stabilization entropy
and the $\varepsilon$-entropy of the feedback $K$ satisfy%
\[
h_{\mathrm{s}}(2\varepsilon,2\zeta,\Gamma,\{0\})\leq h_{\mathrm{fb}%
}(\varepsilon,\zeta,K,\Gamma)\leq h_{\mathrm{top}}\left(  e^{(A+\alpha
I)\cdot}\right)  =\sum_{i:\ \operatorname{Re}(\lambda_{i})>-\alpha
}\operatorname{Re}\lambda_{i}.
\]
Note that here $2\zeta(r,s)=e^{-\alpha s}2Mr,r,s\geq0$. Since the right hand
side is independent of $\varepsilon$, it actually follows for $\varepsilon
\rightarrow0$ that the practical stabilization entropy and the stabilization
entropy satisfies%
\begin{equation}
h_{\mathrm{ps}}(2\zeta,K,\Gamma,\{0\})\leq h_{\mathrm{s}}(2\zeta
,K,\Gamma,\{0\})\leq h_{\mathrm{fb}}(\zeta,K,\Gamma)\leq\sum
_{i:\ \operatorname{Re}(\lambda_{i})>-\alpha}\operatorname{Re}\lambda_{i}.
\label{lin_1}%
\end{equation}
In Theorem \ref{Theorem6.2} we will show that here equalities hold.

For general nonlinear systems, it seems very hard or impossible to derive
explicit formulas for the $\varepsilon$-entropy of a feedback.

\begin{remark}
The papers Liberzon and Hespanha \cite{LibeH05} and De Persis \cite{DePe06}
use Input-to-State (ISS) stability properties in order to derive stabilizing
encoder/decoder controllers. This condition (Assumption 2 in \cite{LibeH05})
requires (in our notation) that there exists a Lipschitz feedback law $u=k(x)$
which satisfies $k(0)=0$ and renders the closed-loop system input-to-state
stable with respect to measurement errors. This means that there exist $\mu
\in\mathcal{KL}$ and $\gamma\in\mathcal{K}_{\infty}$(i.e., $\gamma
:[0,\infty)\rightarrow\lbrack0,\infty)$ is continuous, strictly increasing,
and unbounded with $\gamma(0)=0$) such that for every initial state $x(t_{0})$
and every piecewise continuous signal $e$ the corresponding solution of the
system $\dot{x}=f(x,k(x+e))$ satisfies%
\[
\left\Vert x(t)\right\Vert \leq\mu(\left\Vert x(t_{0})\right\Vert
,t-t_{0})+\gamma(\sup_{s\in\lbrack t_{0},t]}\left\Vert e(s)\right\Vert )\text{
for all }t\geq t_{0}.
\]
The ISS property is used in order to estimate the effect of perturbations on
feedbacks. In Proposition \ref{Proposition_upper}, we have used instead the
entropy property of the feedback $k_{\varepsilon}(\cdot)$.
\end{remark}

\section{Applications\label{Section5}}

In this section we present several examples illustrating practical
stabilization properties and estimates for the corresponding entropies. For
linear control systems we show that the practical stabilization entropy and
the stabilization entropy coincide and they are characterized by a spectral
property. This uses inequality (\ref{lin_1}). Then two scalar examples are
discussed, where quadratic feedbacks and piecewise linear feedbacks, resp.,
only lead to practical stabilization properties. For these examples and a
similar higher dimensional system we estimate the $\varepsilon$-practical
stabilization entropy using the results from Section \ref{Section3}$.$

\subsection{Linear systems\label{Subsection5.1}}

In this subsection, the practical stabilization entropy is determined for
linear control systems in $\mathbb{R}^{d}$ of the form
\begin{equation}
\dot{x}(t)=Ax(t)+Bu(t),\ \ u(t)\in U, \label{6.1}%
\end{equation}
with matrices $A\in\mathbb{R}^{d\times d}$ and $B\in\mathbb{R}^{d\times m}$
and control range $U=\mathbb{R}^{m}$ containing the origin. The next theorem
characterizes the practical stabilization entropy about the equilibrium $x=0$
for linear control systems.

\begin{theorem}
\label{Theorem6.2}Consider a linear control system of the form (\ref{6.1})
with $0\in U$. Assume that there are $M,\alpha>0$ such that for all initial
values $0\not =x\in\mathbb{R}^{d}$ there is a control $u\in\mathcal{U}$ with%
\begin{equation}
\left\Vert \varphi(t,x,u)\right\Vert <e^{-\alpha t}M/2\left\Vert x\right\Vert
\text{ for all }t\geq0. \label{6.2}%
\end{equation}
Then for the exponential $\mathcal{KL}$-function $\zeta(r,s)=e^{-\alpha t}Mr$
the $\varepsilon$-practical stabilization entropy, the practical stabilization
entropy, and the stabilization entropy satisfy for every compact subset
$\Gamma$ with nonvoid interior
\[
h_{\mathrm{ps}}(\varepsilon,\zeta,\Gamma,\{0\})=h_{\mathrm{ps}}(\zeta
,\Gamma,\{0\})=h_{\mathrm{s}}(\zeta,\Gamma,\{0\})=\sum_{\operatorname{Re}%
\lambda_{i}>-\alpha}(\alpha+\operatorname{Re}\lambda_{i}).
\]
Here summation is over all eigenvalues $\lambda_{i}$ of $A$, counted according
to their algebraic multiplicity, with $\operatorname{Re}\lambda_{i}>-\alpha$.
\end{theorem}

\begin{proof}
One easily proves that assumption (\ref{6.2}) holds if and only if there is a
feedback $K$ such that all eigenvalues of $A+BK$ satisfy $\operatorname{Re}%
\lambda_{j}<-\alpha$. Thus condition (\ref{eps3}) holds and it follows from
(\ref{lin_1}) that%
\begin{equation}
h_{\mathrm{ps}}(\zeta,K,\Gamma,\{0\})\leq h_{\mathrm{s}}(\zeta,K,\Gamma
,\{0\})\leq h_{\mathrm{fb}}(\zeta/2,K,\Gamma)\leq\sum
\nolimits_{i:\ \operatorname{Re}(\lambda_{i})>-\alpha}\operatorname{Re}%
\lambda_{i}. \label{6.3}%
\end{equation}
For the proof of the converse inequalities note that $f(x,u)=Ax+Bu$ satisfies
\[
\mathrm{div}_{x}f(x,u)=\mathrm{tr}f_{x}(x,u)=\mathrm{tr}A=\sum_{i=1}%
^{d}\lambda_{i}=\sum_{i=1}^{d}\operatorname{Re}\lambda_{i}.
\]
Theorem \ref{Theorem_lower}(ii) can be applied to the system obtained by the
projection\ $\pi$ to the sum of the real generalized eigenspaces for all
eigenvalues with real part larger than $-\alpha$ along the subspace
corresponding to the sum of the other generalized eigenspaces. Then the
$\varepsilon$-practical stabilization entropy of this projected system is
bounded below by
\[
\alpha\dim(\pi(\mathbb{R}^{d}))+\sum_{\operatorname{Re}\lambda_{i}>-\alpha
}\operatorname{Re}\lambda_{i}=\sum_{\operatorname{Re}\lambda_{i}>-\alpha
}(\alpha+\operatorname{Re}\lambda_{i}).
\]
The equality follows, since the eigenvalues are counted according to their
algebraic multiplicity. Since practically $(\tau,\varepsilon,\zeta
,\Gamma,\{0\})$-spanning sets for the system in $\mathbb{R}^{d}$ yield
practically $(\tau,\varepsilon,\zeta,\pi(\Gamma),\{0\})$-spanning sets for the
projected system and $\pi(\Gamma)$ has nonvoid interior, it follows that
$h_{\mathrm{ps}}(\varepsilon,\zeta,\Gamma,\{0\})\geq\sum_{i:\operatorname{Re}%
\lambda_{i}>-\alpha}(\alpha+\operatorname{Re}\lambda_{i})$. Together with
(\ref{6.3}) the assertion follows.
\end{proof}

\begin{remark}
The characterization of stabilization entropy in Theorem \ref{Theorem6.2} has
already been proved in Colonius \cite[Lemma 4.1 and Theorem 4.2]{Colo12b}. The
proof above is a considerable simplification.
\end{remark}

\subsection{A scalar example with quadratic feedback\label{Subsection5.2}}

In this subsection we discuss a scalar example, where only practical
stabilization properties can be used. Our strategy is to construct quadratic
feedbacks such that the closed loop systems has, in addition to the unstable
equilibrium at the origin, a stable equilibrium arbitrarily close to the
origin. Thus for every $\varepsilon>0$ the feedback systems are $\varepsilon
$-practically stable in the sense of Definition \ref{Definition2.1}, and we
use the estimates for entropy from Section \ref{Section3}$.$

Consider the scalar control system given by%
\begin{equation}
\dot{x}=f(x,u)=\lambda x+\alpha_{0}x^{2}+\beta_{0}xu+\gamma_{0}u^{2},
\label{open}%
\end{equation}
where $\lambda>0$ and $\alpha_{0},\beta_{0},\gamma_{0}$ with $\gamma_{0}%
\not =0$ are real parameters and the controls take values $u(t)\in
U\subset\mathbb{R}$.

For system (\ref{open}) the origin $x=0$ is an equilibrium corresponding to
$u=0$ if $0\in U$. For $x=0$ the right hand side of (\ref{open}) is given by
$f(0,u)=\gamma_{0}u^{2}$. For $\gamma_{0}>0$ one has $f(0,u)>0$ for all
$0\not =u\in\mathbb{R}$ and for $\gamma_{0}<0$ one has $f(0,u)<0$ for all
$0\not =u\in\mathbb{R}$. Hence the system is not controllable around the
origin. By Brockett's necessary condition (cf. Sontag \cite[Theorem
22]{Sont98}) it is not locally $C^{1}$ stabilizable. Hence for $\gamma_{0}>0$
stabilization can only be expected for initial values in $(-\infty,0)$ and for
$\gamma_{0}<0$ for initial values in $(0,\infty)$. \medskip

For quadratic feedbacks of the form
\begin{equation}
k_{\mathrm{quad}}(x)=kx+xqx\text{ with }k,q\in\mathbb{R}, \label{quadratic}%
\end{equation}
the closed loop system is%
\begin{align}
\dot{x}  &  =\lambda x+\alpha_{0}x^{2}+\beta_{0}x(kx+qx^{2})+\gamma
_{0}(kx+qx^{2})^{2}\nonumber\\
&  =\lambda x+(\alpha_{0}+\beta_{0}k+\gamma_{0}k^{2})x^{2}+q(\beta_{0}%
+2\gamma_{0}k)x^{3}+\gamma_{0}q^{2}x^{4}. \label{closed_loop1}%
\end{align}
We denote the solutions of this equation by $\psi(t,x_{0};k,q)$ on their
existence intervals. The following theorem shows that with quadratic feedback
(\ref{quadratic}) system (\ref{open}) can be made $\varepsilon$-practically
stable with exponential rate $\alpha\in(0,3\lambda)$, where the constant $M$
in $\zeta(r,s)=e^{-\alpha s}Mr$ depends on $\varepsilon$ (and $\Gamma$), cf.
Definition \ref{Definition2.1}.

\begin{theorem}
\label{Theorem_equilibria1}Consider system (\ref{open}) with quadratic
feedback (\ref{quadratic}) and $\Lambda=\{0\}$. Fix $\varepsilon>0$ and let
the control range be either $U_{\varepsilon}^{+}=[0,\rho(\varepsilon)]$ or
$U_{\varepsilon}^{-}=[-\rho(\varepsilon),0]$ with $\rho(\varepsilon)$ large
enough. If $\gamma_{0}<0$ consider initial values in a compact set
$\Gamma=\Gamma^{+}\subset(0,\infty)$, if $\gamma_{0}>0$ consider a compact set
$\Gamma=\Gamma^{-}\subset(-\infty,0)$.

Then for every $\alpha\in(0,3\lambda)$ there are $k,q\in\mathbb{R}$ such that
for $\zeta_{\varepsilon}(r,s)=e^{-\alpha s}M(\varepsilon)r,\allowbreak
r,s\geq0$, with $M(\varepsilon)\geq1$, the closed loop system
(\ref{closed_loop1}) is $\varepsilon$-practically $(\zeta_{\varepsilon}%
,\Gamma,\{0\})$-stable.
\end{theorem}

The proof of Theorem \ref{Theorem_equilibria1} is given in the appendix.

Next we estimate the $\varepsilon$-practical stabilization entropy. The
control ranges will vary, hence we add this argument in the notation for the
entropy. In Theorem \ref{Theorem5.3}(ii) we employ the modified notion for
non-compact control ranges in Remark \ref{Remark_unboundedU}.

\begin{theorem}
\label{Theorem5.3}Consider system (\ref{open}) with quadratic feedback
(\ref{quadratic}) and let the assumptions of Theorem \ref{Theorem_equilibria1}
be satisfied.

(i) For $\varepsilon>0$ let the exponential $\mathcal{KL}$-function
$\zeta_{\varepsilon}$ be given by Theorem \ref{Theorem_equilibria1}. Then the
$\varepsilon$-practical stabilization entropy satisfies%
\[
h_{\mathrm{ps}}(2\varepsilon,\zeta_{\varepsilon},\Gamma,\{0\},U_{\varepsilon
}^{\pm})\leq\max\left\{  \left\vert f_{x}(x,u)\right\vert \left\vert (x,u)\in
P_{\varepsilon}\times U_{\varepsilon}^{\pm}\right.  \right\}  <\infty,
\]
where $\Gamma=\Gamma^{+}$ for $\gamma_{0}<0$ and $\Gamma=\Gamma^{-}$ for
$\gamma_{0}>0$, and%
\[
f_{x}(x,u)=\lambda+2\alpha_{0}x+\beta_{0}u\text{ and }P_{\varepsilon}=\left\{
x\in\Gamma\left\vert \left\vert x\right\vert \leq M\max\nolimits_{y\in\Gamma
}\left\vert y\right\vert +\varepsilon(M+1)\right.  \right\}  .
\]

(ii) Suppose that the Lebesgue measure of $\Gamma$ as in (i) is positive, and
either $\beta_{0}>0$ and the control range is $U=U^{+}=[0,\infty)$, or
$\beta_{0}<0$ and the control range is $U=U^{-}=(-\infty,0]$. Assume
$\mathrm{sign}(\gamma_{0})=-\mathrm{sign}(\beta_{0})$. Then for every
$\alpha\in(0,3\lambda)$ and $\zeta(r,s)=e^{-\alpha s}Mr$ with $M\geq1$, the
$\varepsilon$-practical stabilization entropy and the practical stabilization
entropy satisfy%
\[
\infty\geq h_{\mathrm{ps}}(\varepsilon,\zeta,\Gamma,\{0\},U)\geq\alpha
+\lambda-3\left\vert \alpha_{0}\right\vert \varepsilon\text{ and }\infty\geq
h_{\mathrm{ps}}(\zeta,\Gamma,\{0\},U)\geq\alpha+\lambda.
\]
For $\zeta_{\varepsilon}$ given by Theorem \ref{Theorem_equilibria1}, one has
$h_{\mathrm{ps}}(\varepsilon,\zeta_{\varepsilon},\Gamma,\{0\},U)<\infty$.
\end{theorem}

\begin{proof}
(i) Fix $\alpha\in(0,3\lambda)$. Theorem \ref{Theorem_equilibria1} shows for
every $x_{0}\in\Gamma$ there is a control $u(t)=k_{\mathrm{quad}}(\psi
(t,x_{0};k,q)),t\geq0$, with values in $U_{\varepsilon}^{\pm}$ such that%
\[
d(\varphi(t,x_{0},u),\{0\})<\zeta_{\varepsilon}(d(x_{0},\{0\}),t)+\varepsilon
\text{ for all }t\geq0.
\]
Then Theorem \ref{Theorem_upper}(i) yields the upper bound%
\[
h_{\mathrm{ps}}(2\varepsilon,\zeta_{\varepsilon},\Gamma,\{0\},U_{\varepsilon
}^{\pm})\leq L_{\varepsilon}=\max\left\{  \left\vert f_{x}(x,u)\right\vert
\left\vert (x,u)\in P_{\varepsilon}\times U_{\varepsilon}^{\pm}\right.
\right\}  ,
\]
where $f_{x}(x,u)$ and $P_{\varepsilon}$ are as stated in the assertion. This
proves (i).

(ii) Recall the modified notion of $\varepsilon$-practical stabilization
entropy for non-compact control ranges from Remark \ref{Remark_unboundedU},%
\[
h_{\mathrm{ps}}(\varepsilon,\zeta,\Gamma,\{0\},U)=\inf_{K}h_{\mathrm{ps}%
}(\varepsilon,\zeta,\Gamma^{\pm},\{0\},U\cap K),
\]
where the infimum is taken over all compact subset $K\subset\mathbb{R}^{m}$.
Let $K_{0}$ be compact with
\begin{align*}
\inf_{K}h_{\mathrm{ps}}(\varepsilon,\zeta,\Gamma,\{0\},U\cap K)  &  \geq
h_{\mathrm{ps}}(\varepsilon,\zeta,\Gamma,\{0\},U\cap K_{0})-\left\vert
\alpha_{0}\right\vert \varepsilon\\
&  \geq h_{\mathrm{ps}}(\varepsilon,\zeta,\Gamma,\{0\},U_{\varepsilon}^{\pm
})-\left\vert \alpha_{0}\right\vert \varepsilon,
\end{align*}
if $\rho(\varepsilon)>0$ is large enough such that $U\cap K_{0}\subset
U_{\varepsilon}^{+}$ if $\beta_{0}>0$ and $U\cap K_{0}\subset U_{\varepsilon
}^{-}$ if $\beta_{0}<0$. By Theorem \ref{Theorem_lower}(ii) we obtain%
\begin{align*}
h_{\mathrm{ps}}(\varepsilon,\zeta,\Gamma,\{0\},U)  &  \geq\alpha
+\min_{\left\vert x\right\vert \leq\varepsilon,u\in U_{\varepsilon}^{\pm}%
}f_{x}(x,u)-\left\vert \alpha_{0}\right\vert \varepsilon\\
&  =\alpha+\lambda+\min_{\left\vert x\right\vert \leq\varepsilon,u\in
U_{\varepsilon}^{\pm}}(2\alpha_{0}x+\beta_{0}u)-\left\vert \alpha
_{0}\right\vert \varepsilon\\
&  \geq\alpha+\lambda-3\left\vert \alpha_{0}\right\vert \varepsilon+\min_{u\in
U_{\varepsilon}^{\pm}}(\beta_{0}u).
\end{align*}
For $\gamma_{0}>0,\beta_{0}<0$ the control range is $U_{\varepsilon}%
^{-}=[-\rho(\varepsilon),0]$, and we get $\min_{u\in U_{\varepsilon}^{-}%
}(\beta_{0}u)=0$. For $\gamma_{0}<0,\beta_{0}>0$ the control range is
$U_{\varepsilon}^{+}=[0,\rho(\varepsilon)]$, and we get $\min_{u\in
U_{\varepsilon}^{+}}(\beta_{0}u)=\min_{u\in\lbrack0,\rho(\varepsilon)]}%
(\beta_{0}u)=0$. This proves the lower bound on the $\varepsilon$-practical
stabilization entropy. The assertion for the practical stabilization entropy
follows for $\varepsilon\rightarrow0$. The final assertion follows from (i) by
choosing $M(\varepsilon)$ from Theorem \ref{Theorem_equilibria1} and noting
that $h_{\mathrm{ps}}(\varepsilon,\zeta_{\varepsilon},\Gamma,\{0\},U^{\pm
})\leq h_{\mathrm{ps}}(\varepsilon,\zeta_{\varepsilon},\Gamma
,\{0\},U_{\varepsilon}^{\pm})$. This completes the proof of assertion (ii).
\end{proof}

\begin{remark}
Observe that in Theorem \ref{Theorem5.3}(i) the upper bound $L_{\varepsilon}$
converges to $\infty$ for $\varepsilon\rightarrow0$ if $\rho(\varepsilon
)\rightarrow\infty$ and $\beta_{0}\not =0$.
\end{remark}

\subsection{A scalar example with piecewise linear
feedback\label{Subsection5.3}}

Consider the scalar system given by%
\begin{equation}
\dot{x}=f(x,u)=\lambda x+\alpha_{0}x^{2}+\beta_{0}xu+\gamma_{0}u^{2}%
+\alpha_{1}x^{3}+\beta_{1}x^{2}u+\gamma_{1}xu^{2}+\eta_{1}u^{3}, \label{cubic}%
\end{equation}
where $\lambda>0,\alpha_{0},\beta_{0},\gamma_{0},\alpha_{1},\beta_{1}%
,\gamma_{1}$, and $\eta_{1}$ are real parameters with $\gamma_{0},\eta
_{1}\not =0$, and the controls take values $u(t)\in U\subset\mathbb{R}$.

We follow an approach in Hamzi and Krener \cite{HamzK03} to construct a
piecewise linear feedback, such that the closed loop systems has, in addition
to the unstable equilibrium at the origin, two stable equilibria arbitrarily
close to the origin. Then we evaluate the bounds for entropy from Section
\ref{Section3}.

The origin $x=0$ is an equilibrium corresponding to $u=0$. For $x=0$ the right
hand side of (\ref{cubic}) is given by $f(0,u)=\gamma_{0}u^{2}+\eta_{1}u^{3}$.
If the control range is $U=[-\rho,\rho]\subset\mathbb{R}$ with $\rho>0$ large
enough, there are control values $u_{1},u_{2}\in U$ with $f(0,u_{1})>0$ and
$f(0,u_{2})<0$, hence the system is controllable around the origin.

First we will show that system (\ref{cubic}) is practically stabilizable about
the origin using a piecewise linear feedback with $k_{1},k_{2}\in\mathbb{R}$
of the form%
\begin{equation}
k(x)=\left\{
\begin{array}
[c]{ccc}%
k_{1}x & \text{for} & x\geq0\\
k_{2}x & \text{for} & x\leq0
\end{array}
\right.  . \label{plinear}%
\end{equation}
We denote the solutions of the feedback system by $\psi(t,x_{0};k_{1}%
,k_{2}),t\geq0$.

\begin{theorem}
\label{Theorem_equilibria}Consider system (\ref{cubic}) with piecewise linear
feedback (\ref{plinear}), let $\Gamma\subset\mathbb{R}$ be a compact set of
initial values and $\Lambda=\{0\}$. For all $\varepsilon>0$ and $\alpha>0$
there is $M=M(\varepsilon,\alpha)\geq1$ such that the $\mathcal{KL}$-function
$\zeta_{\varepsilon}(r,s)=e^{-\alpha s}Mr,r,s\geq0$, satisfies the following
property. If the control range is $U_{\varepsilon}=[-\rho(\varepsilon
),\rho(\varepsilon)]$ with $\rho(\varepsilon)$ large enough, then there are
$k_{1},k_{2}\in\mathbb{R}$ such that the feedback system is $\varepsilon
$-practically $(\zeta,\Gamma,\{0\})$-stable.
\end{theorem}

The proof of Theorem \ref{Theorem_equilibria} is given in the appendix.

Next we estimate the $\varepsilon$-practical stabilization entropy. In Theorem
\ref{Theorem_unbounded}(ii) we employ the modified notion for non-compact
control ranges in Remark \ref{Remark_unboundedU}.

\begin{theorem}
\label{Theorem_unbounded}Consider system (\ref{cubic}) with piecewise linear
feedback (\ref{plinear}) and let the assumptions of Theorem
\ref{Theorem_equilibria} be satisfied.

(i) For $\varepsilon>0$ let the $\mathcal{KL}$-function $\zeta_{\varepsilon}$
be given by Theorem \ref{Theorem_equilibria}. Then the $\varepsilon$-practical
stabilization entropy satisfies%
\[
h_{\mathrm{ps}}(2\varepsilon,\zeta_{\varepsilon},\Gamma,\{0\},U_{\varepsilon
})\leq L_{\varepsilon}:=\max\left\{  \left\vert f_{x}(x,u)\right\vert
\left\vert (x,u)\in P_{\varepsilon}\times U_{\varepsilon}\right.  \right\}
<\infty,
\]
where%
\begin{align*}
f_{x}(x,u)  &  =\lambda+2\alpha_{0}x+\beta_{0}u+3\alpha_{1}x^{2}+2\beta
_{1}xu+2\gamma_{1}u^{2},\\
P_{\varepsilon}  &  =\left\{  x\in\Gamma\left\vert \left\vert x\right\vert
\leq M\max_{y\in\Gamma}\left\vert y\right\vert +\varepsilon(M+1)\right.
\right\}  .
\end{align*}

(ii) Suppose that the Lebesgue measure of $\Gamma$ is positive and $\gamma
_{1}>0,\beta_{0}\not =0$, and the control range is $U=\mathbb{R}$. Then for a
$\mathcal{KL}$-function $\zeta=e^{-\alpha s}Mr,r,s\geq0$, with $\alpha>0$ and
$M\geq1$ the $\varepsilon$-practical stabilization entropy and the practical
stabilization entropy satisfy%
\begin{align*}
\infty &  \geq h_{\mathrm{ps}}(\varepsilon,\zeta,\Gamma,\{0\},U)\geq
\alpha+\lambda-3\left\vert \alpha_{0}\right\vert \varepsilon-3\left\vert
\alpha_{1}\right\vert \varepsilon^{2}-\frac{1}{4\gamma_{1}}(\beta
_{0}+2\mathrm{sign}(\beta_{0})\left\vert \beta_{1}\right\vert \varepsilon
)^{2},\\
\infty &  \geq h_{\mathrm{ps}}(\zeta,\Gamma,\{0\},U)\geq\alpha+\lambda
-\frac{\beta_{0}^{2}}{4\gamma_{1}}.
\end{align*}
For $\zeta_{\varepsilon}$ given by Theorem \ref{Theorem_equilibria} one has
$h_{\mathrm{ps}}(\varepsilon,\zeta_{\varepsilon},\Gamma,\{0\},U)<\infty$.
\end{theorem}

\begin{proof}
(i) By Theorem \ref{Theorem_equilibria} for every $x_{0}\in\Gamma$ the control
$u(t)=u(\psi(t,x_{0};k_{1},k_{2})),\allowbreak t\geq0$, with values in
$U_{\varepsilon}$ yields%
\[
d(\varphi(t,x_{0},u),\Lambda)<\zeta(d(x_{0},\Lambda)+\varepsilon
,t)+\varepsilon\text{ for all }t\geq0.
\]
Then Theorem \ref{Theorem_upper}(i) yields the upper bound%
\[
h_{\mathrm{ps}}(2\varepsilon,\zeta,\Gamma,\{0\},U_{\varepsilon})\leq
L_{\varepsilon}:=\max\left\{  \left\vert f_{x}(x,u)\right\vert \left\vert
(x,u)\in P_{\varepsilon}\times U_{\varepsilon}\right.  \right\}  ,
\]
where $f_{x}(x,u)$ and $P_{\varepsilon}$ are as stated in the assertion. This
proves (i).

(ii) As in the proof of Theorem \ref{Theorem5.3}\textbf{ }we use that for
$\rho(\varepsilon)$ large enough%
\[
h_{\mathrm{ps}}(\varepsilon,\zeta,\Gamma,\Lambda,U)=\inf_{K}h_{\mathrm{ps}%
}(\varepsilon,\zeta,\Gamma,\Lambda,U\cap K)\geq h_{\mathrm{ps}}(\varepsilon
,\zeta,\Gamma,\Lambda,U_{\varepsilon})-\left\vert \alpha_{0}\right\vert
\varepsilon.
\]
Using the lower estimate provided by Theorem \ref{Theorem_lower}(ii) we obtain%
\begin{align*}
&  h_{\mathrm{ps}}(\varepsilon,\zeta,\Gamma,\{0\},U)\\
&  \geq\alpha+\min_{\left\vert x\right\vert \leq\varepsilon,u\in
U_{\varepsilon}}f_{x}(x,u)-\left\vert \alpha_{0}\right\vert \varepsilon\\
&  =\alpha+\lambda+\min_{\left\vert x\right\vert \leq\varepsilon,u\in
U_{\varepsilon}}\left\{  2\alpha_{0}x+3\alpha_{1}x^{2}+\left(  \beta
_{0}+2\beta_{1}x\right)  u+2\gamma_{1}u^{2}\right\}  -\left\vert \alpha
_{0}\right\vert \varepsilon\\
&  \geq\alpha+\lambda+\min_{\left\vert x\right\vert \leq\varepsilon}%
\{2\alpha_{0}x+3\alpha_{1}x^{2}\}+\min_{\left\vert x\right\vert \leq
\varepsilon}\min_{u\in U_{\varepsilon}}\left\{  \left(  \beta_{0}+2\beta
_{1}x\right)  u+2\gamma_{1}u^{2}\right\}  -\left\vert \alpha_{0}\right\vert
\varepsilon.
\end{align*}
Clearly, $\min_{\left\vert x\right\vert \leq\varepsilon}\{2\alpha_{0}%
x+3\alpha_{1}x^{2}\}-\left\vert \alpha_{0}\right\vert \varepsilon
\geq-3\left\vert \alpha_{0}\right\vert \varepsilon-3\left\vert \alpha
_{1}\right\vert \varepsilon^{2}$ and for the parabola $\left(  \beta
_{0}+2\beta_{1}x\right)  u+\gamma_{1}u^{2},u\in\mathbb{R}$, with $\gamma
_{1}>0$, the minimum is attained in $u=-\frac{\beta_{0}+2\beta_{1}x}%
{2\gamma_{1}}$. Hence for $\varepsilon>0$,%
\begin{align*}
\min_{\left\vert x\right\vert \leq\varepsilon}\min_{u\in U_{\varepsilon}%
}\left\{  \left(  \beta_{0}+2\beta_{1}x\right)  u+2\gamma_{1}u^{2}\right\}
&  =\min_{\left\vert x\right\vert \leq\varepsilon}\left(  -\frac{(\beta
_{0}+2\beta_{1}x)^{2}}{2\gamma_{1}}+\frac{(\beta_{0}+2\beta_{1}x)^{2}}%
{4\gamma_{1}}\right) \\
&  =-\frac{1}{4\gamma_{1}}\max_{\left\vert x\right\vert \leq\varepsilon}%
(\beta_{0}+2\beta_{1}x)^{2}\\
&  \geq-\frac{1}{4\gamma_{1}}(\beta_{0}+2\text{\textrm{sign}(}\beta
_{0})\left\vert \beta_{1}\right\vert \varepsilon)^{2}.
\end{align*}
Together this yields the lower estimate for $h_{\mathrm{ps}}(\varepsilon
,\zeta,\Gamma,\{0\},U)$. The estimate for $h_{\mathrm{ps}}(\zeta
,\Gamma,\{0\},U)$ follows by taking the limit for $\varepsilon\rightarrow0$.
The final assertion is a consequence of (i). This completes the proof of
assertion (ii).
\end{proof}

\subsection{A higher dimensional example\label{Subsection5.4}}

The following system is a generalization of the system in Subsection
\ref{Subsection5.2} by connecting it with a chain of integrators. This system
occurs in a quadratic normal form, cf. Krener, Kang, and Chang \cite[Theorem
2.1]{KrenKC04}. We will rely on a practical stabilization result due to Hamzi
and Krener \cite{HamzK03}. Consider the control system in $\mathbb{R}^{d}$
given by%
\[
\dot{x}_{1}=\lambda x_{1}+\alpha_{0}x_{1}^{2}+\beta_{0}x_{1}x_{2}+\sum
_{j=2}^{d}\gamma_{j}x_{j}^{2},\,\dot{x}_{2}=x_{3},\ldots,\dot{x}_{d}=u,
\]
where $\lambda>0$ and $\alpha_{0},\beta_{0},\gamma_{2}$ with $\gamma_{2}%
\not =0$ are real parameters and the controls take values $u(t)\in
U\subset\mathbb{R}$. Let%
\[
A_{2}=\left[
\begin{array}
[c]{ccccc}%
0 & 1 & 0 & \cdots & 0\\
0 & 0 & 1 & \cdots & 0\\
\vdots &  &  & \ddots & \vdots\\
0 & 0 & 0 & \cdots & 1\\
0 & 0 & 0 & \cdots & 0
\end{array}
\right]  ,\ B_{2}=\left[
\begin{array}
[c]{c}%
0\\
0\\
\vdots\\
0\\
1
\end{array}
\right]  ,
\]
and abbreviate $z=(x_{2},\ldots,x_{d})^{\top}$. Then we may write the system
as%
\begin{align}
\dot{x}_{1}  &  =\lambda x_{1}+\alpha_{0}x_{1}^{2}+\beta_{0}x_{1}x_{2}%
+\sum\nolimits_{j=2}^{d}\gamma_{j}x_{j}^{2}\label{5.4a}\\
\dot{z}  &  =A_{2}z+B_{2}u.\nonumber
\end{align}
We use linear feedback of the form%
\begin{equation}
k(x_{1},z)=k_{1}x_{1}+K_{2}z \label{5.4b}%
\end{equation}
with $k_{1}\in\mathbb{R}$ and choose $K_{2}\in\mathbb{R}^{1\times(d-1)}$ such
that $A_{2}+B_{2}K_{2}$ is stable. The feedback system becomes%
\begin{align}
\dot{x}_{1}  &  =\lambda x_{1}+\alpha_{0}x_{1}^{2}+\beta_{0}x_{1}x_{2}%
+\sum\nolimits_{j=2}^{d}\gamma_{j}x_{j}^{2}\label{5.4c}\\
\dot{z}  &  =\left(  A_{2}+B_{2}K_{2}\right)  z+B_{2}k_{1}x_{1}.\nonumber
\end{align}
The following theorem shows a practical stabilizability result.

\begin{theorem}
\label{Theorem_HK}Consider system (\ref{5.4a}) with linear feedback
(\ref{5.4b}) and $\Lambda=\{0\}$. Suppose that $\lambda>0$ is sufficiently
small. Fix $\varepsilon>0$ and let $\rho(\varepsilon)>0$ be large enough. If
$\gamma_{2}<0$ let the control range be $U_{\varepsilon}^{+}=[0,\rho
(\varepsilon)]$ and consider initial values in a compact set $\Gamma
=\Gamma^{+}\subset(0,\infty)$. If $\gamma_{2}>0$ let the control range be
$U_{\varepsilon}^{-}=[-\rho(\varepsilon),0]$ and consider initial values in a
compact set $\Gamma=\Gamma^{-}\subset(-\infty,0)$.

Then there is a $\mathcal{KL}$-function $\zeta_{\varepsilon}$ such that the
closed loop system (\ref{closed_loop1}) is $\varepsilon$-practically
$(\zeta_{\varepsilon},\Gamma,\{0\})$-stable.
\end{theorem}

\begin{proof}
This follows from Hamzi and Krener \cite[proof of Theorem 3, pp.
44-47]{HamzK03}. Here a center manifold reduction is used to show the
following. If $\gamma_{2}<0$ there is $k_{1}<0$ such that the feedback system
(\ref{5.4c}) has an equilibrium $e^{+}$ with positive first component and
domain of attraction including $\Gamma^{+}$. If $\gamma_{2}>0$ there is
$k_{1}>0$ such that the feedback system (\ref{5.4c}) has an equilibrium
$e^{-}$ with negative first component and domain of attraction including
$\Gamma^{-}$ and corresponding $\mathcal{KL}$-function $\zeta_{\varepsilon}$.
Choosing $\left\vert k_{1}\right\vert $ large enough, these equilibria are
arbitrarily close to the origin \ This implies the assertion.
\end{proof}

\begin{theorem}
Consider system (\ref{5.4a}) with linear feedback (\ref{5.4b}) and let the
assumptions of Theorem \ref{Theorem_HK} be satisfied. For every $\varepsilon
>0$ define
\[
L_{\varepsilon}:=\max\left(  1,\max\nolimits_{x\in P_{\varepsilon}}\left\{
\lambda+2\left\vert \alpha_{0}\right\vert x_{1}+\left\vert \beta
_{0}\right\vert x_{2}+2\sum\nolimits_{j=2}^{d}\left\vert \gamma_{j}\right\vert
x_{j}\right\}  \right)  ,
\]
where $P_{\varepsilon}:=\left\{  x\in\mathbb{R}^{d}\left\vert \left\Vert
x\right\Vert \leq\zeta_{\varepsilon}(\max_{y\in\Gamma}\left\Vert y\right\Vert
+\varepsilon,0)+\varepsilon\right.  \right\}  $. Then for every $\varepsilon
>0$ the $\varepsilon$-practical stabilization entropy satisfies%
\[
\lambda-(2\left\vert \alpha_{0}\right\vert +\left\vert \beta_{0}\right\vert
)\varepsilon\leq h_{\mathrm{ps}}(\varepsilon,\zeta_{\varepsilon}%
,\Gamma,\{0\})\leq L_{\varepsilon/2}\,\,d.
\]

\end{theorem}

\begin{proof}
Denoting the right hand side of (\ref{5.4a}) by $f(x,u)$ one finds%
\[
f_{x}(x,u)=\left[
\begin{array}
[c]{cccc}%
\lambda+2\alpha_{0}x_{1}+\beta_{0}x_{2} & \beta_{0}x_{1}+2\gamma_{2}x_{2} &
\cdots & 2\gamma_{d}x_{d}\\
0 &  & A_{2} &
\end{array}
\right]  .
\]
Using the max-norm in $\mathbb{R}^{d}$ and $\mathrm{tr}A_{2}=0$, Theorem
\ref{Theorem_lower}(i) yields the lower bound%
\begin{align*}
h_{\mathrm{ps}}(\varepsilon,\zeta_{\varepsilon},\Gamma,\{0\})  &  \geq
\min\left\{  \mathrm{tr}f_{x}(x,u)\left\vert (x,u)\in\mathbf{B}(0,\varepsilon
)\times U_{\varepsilon}\right.  \right\} \\
&  =\min\left\{  \lambda+2\alpha_{0}x_{1}+\beta_{0}x_{2}+\mathrm{tr}%
A_{2}\left\vert x\in\mathbf{B}(0,\varepsilon)\right.  \right\} \\
&  =\lambda-2\left\vert \alpha_{0}\right\vert \varepsilon-\left\vert \beta
_{0}\right\vert \varepsilon.
\end{align*}
Theorem \ref{Theorem_upper}(i) yields the upper bound $L_{\varepsilon/2}\,d$
with $L_{\varepsilon/2}:=\max_{(x,u)\in P_{\varepsilon/2}\times U}\left\Vert
f_{x}(x,u)\right\Vert \allowbreak<\infty$. Using the matrix norm induced by
the max-norm in $\mathbb{R}^{d}$ we get%
\[
\left\Vert f_{x}(x,u)\right\Vert =\max\left\{  \left\vert \lambda+2\alpha
_{0}x_{1}+\beta_{0}x_{2}\right\vert +\sum_{j=2}^{d}\left\vert 2\gamma_{j}%
x_{j}\right\vert ,1\right\}  ,
\]
hence%
\[
\max\left\{  \left\Vert f_{x}(x,u)\right\Vert \left\vert (x,u)\in
P_{\frac{\varepsilon}{2}}\times U_{\frac{\varepsilon}{2}}\right.  \right\}
=\max_{x\in P_{\frac{\varepsilon}{2}}}\left\{  \lambda+2\left\vert \alpha
_{0}\right\vert x_{1}+\left\vert \beta_{0}\right\vert x_{2}+2\sum_{j=2}%
^{d}\left\vert \gamma_{j}\right\vert x_{j}\right\}  .
\]

\end{proof}

\section{Conclusions and open questions\label{Section6}}

In Section \ref{Section3}, we have derived upper and lower bounds for
$\varepsilon$-practical stabilization entropy and practical stabilization
entropy (i.e., in the limit for $\varepsilon\rightarrow0$) based on general
$\mathcal{KL}$-functions $\zeta$, with special attention to exponential
$\mathcal{KL}$-functions. Section \ref{Section4} presents an upper bound for
$\varepsilon$-practical stabilization entropy based on an $\varepsilon
$-entropy notion for feedbacks. In Section \ref{Section5} this is used for
linear control systems in order to prove that the practical stabilization
entropy and the stabilization entropy coincide provided that the system is
stabilizable and to characterize them by a spectral condition. Two scalar
examples are analyzed where quadratic feedbacks and piecewise linear
feedbacks, resp., only lead to $\varepsilon$-practical stabilization for every
$\varepsilon>0$. Here and for a similar higher dimensional system the employed
exponential $\mathcal{KL}$-functions depend on $\varepsilon$ and the upper
bounds diverge for $\varepsilon\rightarrow0$.

Major research problems include the following: Suppose that the considered
control system is $\varepsilon$-practically stabilizable for every
$\varepsilon>0$, but not stabilizable (either by appropriate feedbacks or in
the sense of (\ref{ass2}), where open loop controls are considered). Will the
corresponding $\varepsilon$-practical stabilization entropies diverge for
$\varepsilon\rightarrow0$? It is also not clear to us, when there exist
$\mathcal{KL}$-functions which work for every $\varepsilon>0$. Furthermore,
suppose that the system is stabilizable. Is there a gap between the practical
stabilization entropy and the stabilization entropy? In the linear case,
Theorem \ref{Theorem6.2} shows that both entropy notions coincide. The answer
will be of interest for control devices which only lead to practical
stability, but not to stability. Furthermore, the relations of practical
stabilization entropy to minimal data rates for digital communication channels
merits exploration.

Our results do not yield formulas for practical stabilization entropy. In the
well studied case of invariance entropy, only for hyperbolic control systems
such strong results are available, cf. Kawan and Da Silva \cite{KawaDS16}. In
this context Kawan \cite{Kawa20} shows a lower bound for stabilization in
terms of topological pressure under a uniform hyperbolicity assumption. See
also Kawan \cite{Kawa17}\textbf{ }for a general discussion of hyperbolicity in
the context of control systems. \ However, hyperbolicity conditions are not
directly applicable in our framework, since it is not local (with respect to
$\Lambda$).

\section{Appendix\label{Appendix}}

In this appendix we prove Theorem \ref{Theorem_equilibria1} and Theorem
\ref{Theorem_equilibria}.

\begin{proof}
[\textit{of Theorem \ref{Theorem_equilibria1}}]. Equilibria different from the
trivial equilibrium $x=0$ satisfy%
\begin{equation}
0=\lambda+(\alpha_{0}+\beta_{0}k+\gamma_{0}k^{2})x+q(\beta_{0}+2\gamma
_{0}k)x^{2}+\gamma_{0}q^{2}x^{3}. \label{equilibrium1}%
\end{equation}
We choose the constant $k$ in order to eliminate the quadratic term,%
\begin{equation}
\beta_{0}+2\gamma_{0}k=0,\text{ i.e., }k=-\frac{\beta_{0}}{2\gamma_{0}}.
\label{k}%
\end{equation}
Then it is immediately clear that the properties of the feedback system
(\ref{closed_loop1}) do not depend on the sign of $q$. Furthermore one finds%
\begin{equation}
\alpha_{0}+\beta_{0}k+\gamma_{0}k^{2}=\alpha_{0}-\frac{\beta_{0}^{2}}%
{2\gamma_{0}}+\frac{\beta_{0}^{2}}{4\gamma_{0}}=\frac{4\alpha_{0}\gamma
_{0}-\beta_{0}^{2}}{4\gamma_{0}}, \label{k_next}%
\end{equation}
hence the equilibria are determined by%
\begin{equation}
x^{3}+\frac{4\alpha_{0}\gamma_{0}-\beta_{0}^{2}}{4\gamma_{0}^{2}q^{2}}%
x+\frac{\lambda}{\gamma_{0}q^{2}}=0. \label{x}%
\end{equation}
The solutions of this cubic equation in reduced form $x^{3}+3ax+b=0$ are given
by the classical Cardano formula, cf. e.g. Zwillinger \cite[Subsection
2.3.2]{CRC}. For (\ref{x}) one has $a=\frac{4\alpha_{0}\gamma_{0}-\beta
_{0}^{2}}{12\gamma_{0}^{2}q^{2}}$ and $b=\frac{\lambda}{\gamma_{0}q^{2}}$. If
the discriminant%
\[
D:=4a^{3}+b^{2}=\frac{4}{q^{6}}\left(  \frac{4\alpha_{0}\gamma_{0}-\beta
_{0}^{2}}{12\gamma_{0}^{2}}\right)  ^{3}+\frac{1}{q^{4}}\frac{\lambda^{2}%
}{\gamma_{0}^{2}}>0,
\]
there is a unique real real solution, hence a unique nontrivial equilibrium,
given by%
\[
e(q)=(-\frac{b}{2}+\frac{1}{2}\sqrt{D})^{\frac{1}{3}}+(-\frac{b}{2}-\frac
{1}{2}\sqrt{D})^{\frac{1}{3}}=\left(  -\frac{\lambda}{2\gamma_{0}q^{2}}%
+\frac{1}{2}\sqrt{D}\right)  ^{\frac{1}{3}}+\left(  -\frac{\lambda}%
{2\gamma_{0}q^{2}}-\frac{1}{2}\sqrt{D}\right)  ^{\frac{1}{3}}.
\]
The condition $D>0$ holds for $\left\vert q\right\vert $ large enough. The
dominant term for $\left\vert q\right\vert \rightarrow\infty$ in $D$ is
$\frac{1}{q^{4}}\frac{\lambda^{2}}{\gamma_{0}^{2}}$ and hence the dominant
term in $e(q)$ is%
\begin{equation}
\left(  -\frac{\lambda}{2\gamma_{0}q^{2}}+\frac{1}{2}\frac{\lambda}%
{q^{2}\gamma_{0}}\right)  ^{1/3}+\left(  -\frac{\lambda}{2\gamma_{0}q^{2}%
}-\frac{1}{2}\frac{\lambda}{q^{2}\gamma_{0}}\right)  ^{1/3}=-\frac{1}{q^{2/3}%
}\frac{\lambda^{1/3}}{\gamma_{0}^{1/3}}. \label{5.9}%
\end{equation}
Thus for $\left\vert q\right\vert \rightarrow\infty$ one has $e(q)\rightarrow
0$ with $\left\vert q\right\vert ^{-2/3}$. Let $\left\vert q\right\vert $ be
large enough. Then for $\gamma_{0}>0$ the initial values are taken in
$\Gamma\subset(-\infty,e(q))$ and $e(q)<0$. For $\gamma_{0}<0$, the initial
values are taken in $\Gamma\subset(e(q),\infty)$ and $e(q)>0$.

Next we analyze stability of the equilibria. Since the equilibrium in the
origin is unstable, general properties of scalar autonomous differential
equations imply that $e(q)$ is asymptotically stable with domain of attraction
given by $(-\infty,0)$ if $e(q)<0$ and $(0,\infty)$ if $e(q)>0$. In order to
prove exponential stability, we compute the Jacobian of (\ref{closed_loop1})
with (\ref{k}) and (\ref{k_next})%
\begin{align*}
J(x)  &  =\frac{\partial}{\partial x}\left[  \lambda x+(\alpha_{0}+\beta
_{0}k+\gamma_{0}k^{2})x^{2}+q(\beta_{0}+2\gamma_{0}k)x^{3}+\gamma_{0}%
q^{2}x^{4}\right] \\
&  =\lambda+\frac{4\alpha_{0}\gamma_{0}-\beta_{0}^{2}}{2\gamma_{0}}%
x+4\gamma_{0}q^{2}x^{3}.
\end{align*}
For $x=e(q)$ we obtain by (\ref{5.9}) that for $\left\vert q\right\vert
\rightarrow\infty$ the dominant term in the Jacobian $J(e(q))$ is
\[
\lambda+\frac{4\alpha_{0}\gamma_{0}-\beta_{0}^{2}}{2\gamma_{0}}\left(
-\frac{1}{q^{2/3}}\frac{\lambda^{1/3}}{\gamma_{0}^{1/3}}\right)  -4\gamma
_{0}q^{2}\frac{1}{q^{2}}\frac{\lambda}{\gamma_{0}}=\lambda-\frac{1}{q^{2/3}%
}\frac{4\alpha_{0}\gamma_{0}-\beta_{0}^{2}}{2\gamma_{0}}\frac{\lambda^{1/3}%
}{\gamma_{0}^{1/3}}-4\lambda\rightarrow-3\lambda.
\]
Hence for $\left\vert q\right\vert $ large enough, the equilibrium $e(q)$ is
locally exponentially stable with $\zeta(r,s)=e^{-\alpha s}r$ for
$0<\alpha<3\lambda$.

\textbf{Claim}: There is $M=M(\varepsilon)\geq1$ such that for every
$\alpha\in(0,3\lambda)$ and every $x_{0}\in\Gamma$ the following exponential
estimate holds,%
\begin{equation}
\left\vert \psi(t,x_{0};k,q)-e(q)\right\vert \leq e^{-\alpha t}M\left\vert
x_{0}-e(q)\right\vert \text{ for }t\geq0. \label{exp_est}%
\end{equation}
For the proof we first consider the case $e(q)>0$. Then one can choose
$z=z(q)\in\mathbb{R}$ in the domain of exponential attraction such that
$e(q)<z$, hence%
\[
\psi(t,z;k,q)-e(q)\leq e^{-\alpha t}(z-e(q))\text{ for }t\geq0.
\]
For every $x_{0}\in\Gamma$ there is $T_{x_{0}}>0$ with $\psi(T_{x_{0}}%
,x_{0};k,q)=z$, hence for $t\geq0$%
\begin{equation}
\psi(t+T_{x_{0}},x_{0};k,q)-e(q)=\psi(t,\psi(T_{x_{0}},x_{0}%
;k,q);k,q)-e(q)\leq e^{-\alpha t}\left(  z-e(q)\right)  . \label{Tx}%
\end{equation}
By compactness of $\Gamma$ it follows that $T:=\max_{x_{0}\in\Gamma}T_{x_{0}%
}<\infty$ and hence with $M=e^{3\lambda T}$ it follows for $x_{0}\in\Gamma$%
\[
\psi(t,x_{0};k,q)-e(q)\leq z-e(q)\leq e^{-\alpha T}M\max_{y\in\Gamma
}\{y-e(q)\}\text{ for }t\in\lbrack0,T],
\]
and for $t>T_{x_{0}}$ this yields together with (\ref{Tx}) and (\ref{exp_est}%
),
\begin{align*}
\psi(t,x_{0};k,q)-e(q)  &  =\psi(t-T_{x_{0}},\psi(T_{x_{0}},x_{0}%
;k,q);k,q)-e(q)\leq e^{-\alpha(t-T_{x_{0}})}\left(  z-e(q)\right) \\
&  \leq e^{-\alpha(t-T_{x_{0}})}e^{-\alpha T_{x_{0}}}M\left(  x_{0}%
-e(q)\right)  =e^{-\alpha t}M\left(  x_{0}-e(q)\right)  .
\end{align*}
Note that $T_{x_{0}}$ and $T$ depend on $q$, since the point $z=z(q)$ is taken
in the domain of exponential attraction of $e(q)$, hence depends on $q$. This
entails that $M$ depends on $\varepsilon$, since $\left\vert q\right\vert $ is
taken large enough in dependence on $\varepsilon$.

Analogously, one argues for $e(q)<0$. Thus the \textbf{claim} is proved.

Choosing $\left\vert q\right\vert $ large enough, it follows that
$e(q)\in(-\frac{\varepsilon}{2M},\frac{\varepsilon}{2M})$, hence
\begin{align*}
\left\vert \psi(t,x_{0};k,q)\right\vert  &  \leq\left\vert \psi(t,x_{0}%
;k,q)-e(q)\right\vert +\left\vert e(q)\right\vert \leq e^{-\alpha
t}M\left\vert x_{0}-e(q)\right\vert +\left\vert e(q)\right\vert \\
&  \leq e^{-\alpha t}M\left\vert x_{0}\right\vert +\varepsilon.
\end{align*}
Thus the system is $\varepsilon$-practically $(\zeta_{\varepsilon}%
,\Gamma,\{0\})$-stable with $\zeta_{\varepsilon}(r,s)=e^{-\alpha
s}M(\varepsilon)r$. Our assumptions on the control range $U_{\varepsilon}%
^{\pm}$ guarantee that the values of the quadratic feedback (\ref{quadratic})
can be taken in $U_{\varepsilon}^{\pm}$ for $x_{0}\in\Gamma$. \medskip
\end{proof}

\begin{proof}
[\textit{of Theorem \ref{Theorem_equilibria}}]. We will show that for every
$\varepsilon>0$ there are $k_{1},k_{2}\in\mathbb{R}$ such that, besides the
trivial equilibrium at the origin, the feedback system has two equilibria in
$\left(  -\varepsilon,\varepsilon\right)  $ which are locally exponentially stable.

Let $i\in\{1,2\}$. Any nontrivial equilibrium $x$ must satisfy
\begin{equation}
0=\lambda+[\alpha_{0}+\beta_{0}k_{i}+\gamma_{0}k_{i}^{2}]x+[\alpha_{1}%
+\beta_{1}k_{i}+\gamma_{1}k_{i}^{2}+\eta_{1}k_{i}^{3}]x^{2}.
\label{equilibrium0}%
\end{equation}
Abbreviate%
\[
\Delta_{0,i}(k_{i})=\alpha_{0}+\beta_{0}k_{i}+\gamma_{0}k_{i}^{2}\text{ and
}\Delta_{1,i}(k_{i})=\alpha_{1}+\beta_{1}k_{i}+\gamma_{1}k_{i}^{2}+\eta
_{1}k_{i}^{3}.
\]
Then the solutions of (\ref{equilibrium0}) are%
\begin{equation}
x_{i}^{\pm}(k_{i})=\frac{-\Delta_{0,i}{(k}_{i})\pm\sqrt{\Delta_{0,i}^{2}%
{(k}_{i})-4\lambda\Delta_{1,i}{(k}_{i})}}{2\Delta_{1,i}{(k}_{i})}. \label{x_i}%
\end{equation}
\textbf{Claim: }The feedback system with $\left\vert k_{i}\right\vert $ large
enough and $\mathrm{sign}(k_{i})\allowbreak=-\mathrm{sign}(\eta_{1})$ has
three equilibria given by%
\[
e_{2}(k_{2}):=x_{2}^{+}(k_{2})<0<e_{1}(k_{1}):=x_{1}^{-}(k_{1}).
\]
For the proof of the claim observe first that $x_{i}^{\pm}(k_{i})\in
\mathbb{R}$ if%
\begin{equation}
{\Delta_{0.i}^{2}(k}_{i}){-4\lambda\Delta_{1,i}(k}_{i})>0, \label{A1}%
\end{equation}
and $x_{1}^{\pm}(k_{1})$ is an equilibrium iff it is positive and $x_{2}^{\pm
}(k_{2})$ is an equilibrium iff it is negative. For $k_{i}$ as in the claim it
follows that ${\Delta_{1,i}(k}_{i})<0$. Hence (\ref{A1}) holds and $x_{i}%
^{\pm}(k_{i})\in\mathbb{R}$. For $\left\vert k_{i}\right\vert \rightarrow
\infty$ it follows that $\Delta_{1,i}{(k}_{i})\rightarrow-\infty$ with
$\left\vert k_{i}\right\vert ^{3}$.

We distinguish the following two cases.

- Let $\gamma_{0}>0$. For $\left\vert k_{i}\right\vert \rightarrow\infty$ it
follows that $\Delta_{0,i}(k_{i})\rightarrow\infty$ with $\left\vert
k_{i}\right\vert ^{2}$, hence $x_{i}^{\pm}(k_{i})\rightarrow0$ with
$\left\vert k_{i}\right\vert ^{-1}$, and $x_{i}^{-}(k_{i})>0$ and $x_{i}%
^{+}(k_{i})<0$ for $\left\vert k_{i}\right\vert $ large enough.

- Let $\gamma_{0}<0$. For $\left\vert k_{i}\right\vert \rightarrow\infty$ it
follows that $\Delta_{0,i}(k_{i})\rightarrow-\infty$ with $\left\vert
k_{i}\right\vert ^{2}$, hence $x_{i}^{\pm}(k_{i})\rightarrow0$ with
$\left\vert k_{i}\right\vert ^{-1}$. Again, $x_{i}^{-}(k_{i})>0$ and
$x_{i}^{+}(k_{i})<0$ for $\left\vert k_{i}\right\vert $ large enough.

Thus the claim is proved. Note that for $\left\vert k_{i}\right\vert $ large
enough, the equilibria $e_{i}(k_{i})$ are arbitrarily close to $0$.

Next we analyze the stability properties of the equilibria. Since the
equilibrium in the origin is unstable, it follows from general properties of
scalar autonomous differential equations that $e_{1}(k_{1})$ and $e_{2}%
(k_{2})$ are asymptotically stable with domains of attraction $(0,\infty)$ and
$(-\infty,0)$, respectively. In order to prove exponential stability, we
compute the Jacobian in a nontrivial equilibrium of the feedback system using
the product rule and (\ref{equilibrium0}),%
\begin{align*}
J(x)  &  =\frac{\partial}{\partial x}\left[  x(\lambda+\alpha_{0}x+\beta
_{0}k_{i}x+\gamma_{0}k_{i}^{2}x+\alpha_{1}x^{2}+\beta_{1}x^{2}k_{i}+\gamma
_{1}xk_{i}^{2}x+\eta_{1}k_{i}^{3}x^{2}\right] \\
&  =x\left[  \alpha_{0}+\beta_{0}k_{i}+\gamma_{0}k_{i}^{2}+2(\alpha_{1}%
+\beta_{1}k_{i}+\gamma_{1}k_{i}^{2}+\eta_{1}k_{i}^{3})x\right] \\
&  =x\left[  \Delta_{0,i}(k_{i})+2\Delta_{1,i}(k_{i})x\right]  .
\end{align*}
For $x=x_{1}^{-}(k_{1})=e_{1}(k_{1})>0$ one finds by (\ref{x_i})%
\begin{align*}
J(e_{1}(k_{1}))  &  =e_{1}(k_{1})\left[  \Delta_{0,1}(k_{1})+2\Delta
_{1,1}(k_{1})e_{1}(k_{1})\right] \\
&  =-e_{1}(k_{1})\sqrt{\Delta_{0,1}(k_{1})^{2}-4\lambda\Delta_{1,1}(k_{1})}<0,
\end{align*}
and for $x=x_{2}^{+}(k_{2})=e_{2}(k_{2})<0$ one obtains%
\[
J(e_{2}(k_{2}))=e_{2}(k_{2})\sqrt{\Delta_{0,2}(k_{2})^{2}-4\lambda\Delta
_{1,2}(k_{2})}<0.
\]
Concluding, we have found constants $k_{1},k_{2}$ with $\left\vert
k_{1,2}\right\vert $ large enough such that there are, besides the trivial
equilibrium at the origin, the two equilibria $e_{2}(k_{2})<0<e_{1}(k_{1})$
which are locally exponentially stable with domain of asymptotic attraction
$(-\infty,0)$, and $(0,\infty)$, respectively. Note that for $\left\vert
k_{i}\right\vert \rightarrow\infty$ one has that $J(e_{i}(k_{i}))\rightarrow
-\infty$ with $\left\vert k_{i}\right\vert $. This follows, since $e_{i}%
(k_{i})\rightarrow0$ with $\left\vert k_{i}\right\vert ^{-1}$ and
$\sqrt{\Delta_{0,i}(k_{i})^{2}-4\lambda\Delta_{1,2}(k_{i})}\allowbreak
\rightarrow\infty$ with $\left\vert k_{i}\right\vert ^{2}$. Hence the
exponential rate $\alpha$ can be chosen arbitrarily large for $\left\vert
k_{i}\right\vert $ large enough.

Similarly as in the proof of Theorem \ref{Theorem_equilibria1} one finds for
$\alpha>0$ a constant $M=M(\varepsilon,\alpha)\geq1$ such that the solutions
$\psi(t,x_{0};k_{1},k_{2})$ satisfy for $x_{0}\in\Gamma\cap(0,\infty]$ with
$i=1$ and for $x_{0}\in\Gamma\cap(-\infty,0)$ with $i=2$,
\begin{equation}
\left\vert \psi(t,x_{0};k_{1},k_{2})-e_{i}(k_{i})\right\vert \leq e^{-\alpha
t}M\left\vert x_{0}\right\vert ,t\geq0.\nonumber
\end{equation}
With $\zeta_{\varepsilon}(r,s)=e^{-\alpha s}M(\varepsilon,\alpha)r$ one shows
as in the proof of Theorem \ref{Theorem_equilibria1} that the system is
$\varepsilon$-practically $(\zeta_{\varepsilon},\Gamma,\{0\})$-stable for
$\left\vert k_{1,2}\right\vert $ large enough. The required control values
$k_{i}\psi(t,x_{0},k_{1},k_{2}),x_{0}\in\Gamma$, are in the control range, if
$\rho(\varepsilon)$ is large enough.
\end{proof}

\textbf{Acknowledgments}. We are grateful to Lars Gr\"{u}ne for references
showing that stabilization algorithms for asymptotically stabilizable
equilibria may only lead to practical stabilization, Daniel Liberzon observed
that the proof of \cite[Lemma 4.1, Theorem 4.2]{Colo12b} can be simplified, we
have adapted some of his arguments for the proof of Theorem \ref{Theorem6.2}.


\begin{thebibliography}{99}                                                                                               %


\bibitem {AbrMR88}\textsc{R.~Abraham, J.~Marsden, and T.~Ratiu},
\emph{Manifolds, Tensor Analysis, and Applications}, Springer Verlag, New
York, 1988.

\bibitem {BeJu20}\textsc{G.O.Berger and R.M. Jungers,} \emph{Finite data-rate
feedback stabilization of continuous-time switched linear systems with unknown
switching signal}, arXiv:2009.04715v1, (2020).

\bibitem {05BoiLR}\textsc{V.A. Boichenko, G.A. Leonov, and V. Reitmann,}
\emph{Dimension Theory for Ordinary Differential Equations,} Teubner 2005.

\bibitem {CLS98}\textsc{F.H. Clarke, Y.S. Ledyaev, and R.J. Stern},
\emph{Asymptotic stability and smooth Lyapunov functions}, J. Diff. Equations,
149(1) (1998), pp.69--114.

\bibitem {Colo12b}\textsc{F.~Colonius}, \emph{Minimal bit rates and entropy
for stabilization}, SIAM J. Control Optim., 50 (2012), pp.~2088--3010.

\bibitem {ColoK09a}\textsc{F.~Colonius and C.~Kawan}, \emph{Invariance entropy
for control systems}, SIAM J. Control Optim., 48 (2009), pp.~1701--1721.

\bibitem {DaSiK18a}\textsc{A.~{da Silva} and C.~Kawan}, \emph{Robustness of
critical bit rates for practical stabilization of networked control systems},
Automatica, 93 (2018), pp.~397--406.

\bibitem {DePe06}\textsc{C.~de~Persis}, \emph{Nonlinear stabilizability via
encoded feedback: The case of integral {ISS} systems}, Automatica, 42 (2006), pp.~1813--1816.

\bibitem {Grue}\textsc{L.~Gr{\"{u}}ne}, \emph{Stabilization by sampled and
discrete feedback with positive sampling rate}, in: Stability and
Stabilization of Nonlinear Systems, Lecture Notes in Control and Information
Sciences (LNCIS, volume 246), D. Aeyels, F. Lamnabhi-Lagarrigue, A. van der
Schaft (eds.), pp. 165-182.

\bibitem {HamzK03}\textsc{B.~Hamzi and A.~Krener}, \emph{Practical
stabilization of systems with a fold control bifurcation}, in New Trends in
Nonlinear Dynamics and Control, W.~{Kang et al.}, eds., vol.~295 of LNCIS,
Berlin Heidelberg, 2003, Springer-Verlag, pp.~37--48.

\bibitem {HuanZ18}\textsc{Y.~Huang and X.~Zhong}, \emph{Carath{\'{e}}%
odory{-P}esin structures associated with control systems}, Systems and Control
Letters, 112 (2018), pp.~36--41.

\bibitem {KatH95}\textsc{A.~Katok and B.~Hasselblatt}, \emph{Introduction to
the Modern Theory of Dynamical Systems}, Cambridge University Press, 1995.

\bibitem {Kawa20}\textsc{C.~Kawan}, \emph{Control of chaos with minimal
information transfer}. \newblock arxiv.org/abs/2003.06935.

\bibitem {Kawa17}\leavevmode\vrule height 2pt depth -1.6pt width 23pt,
\emph{Uniformly hyperbolic control theory,} Annual Reviews in Control 44
(2017), pp. 89-96.

\bibitem {Kawa13}\leavevmode\vrule height 2pt depth -1.6pt width 23pt,
\emph{Invariance Entropy for Deterministic Control Systems. An Introduction},
vol.~2089 of Lecture Notes in Mathematics, Springer-Verlag, 2013.

\bibitem {KawaDS16}\textsc{C.~Kawan and A.~Da Silva}, \emph{Invariance entropy
of hyperbolic control sets}, Discrete and Continuous Dynamical Systems, 30 (1)
(2016), pp. 97-136.

\bibitem {KaMP21}\textsc{C. Kawan, A.~Matveev and A.~Pogromsky, }\emph{Remote
state estimation problem: Towards the data-rate limit along} \emph{the avenue
of the second Lyapunov method, }Automatica, 125 (2021), 109467.

\bibitem {KrenKC04}\textsc{A.~J. Krener, W.~Kang, and D.~E. Chang},
\emph{Control bifurcations}, IEEE Trans. Aut. Control, 49 (2004), pp.~1231--1246.

\bibitem {LibeH05}\textsc{D.~Liberzon and J.~Hespanha}, \emph{Stabilization of
nonlinear system with limited information feedback}, IEEE Trans. Aut. Control,
50 (2005), pp.~910--915.

\bibitem {LibeM18}\textsc{D. Liberzon and S. Mitra,} \emph{Entropy and minimal
bit rates for state estimation and control, }IEEE Trans. Aut. Control, 63
(2018), pp. 3330-3344.

\bibitem {MatvP16}\textsc{A.~Matveev and A.~Pogromsky}, \emph{Observation of
nonlinear systems via finite capacity channels: Constructive data rate
limits}, Automatica, 70 (2016), pp.~217--229.

\bibitem {MatvP19}\leavevmode\vrule height 2pt depth -1.6pt width 23pt,
\emph{Observation of nonlinear systems via finite capacity channels. part II:
Restoration entropy and its estimates}, Automatica, 103 (2019), pp.~189--199.

\bibitem {NEMM04}\textsc{G.~Nair, R.~J. Evans, I.~Mareels, and W.~Moran},
\emph{Topological feedback entropy and nonlinear stabilization}, IEEE Trans.
Aut. Control, 49 (2004), pp.~1585--1597.

\bibitem {Sont98}\textsc{E.~D. Sontag}, \emph{Mathematical Control Theory},
Springer-Verlag, 1998. \newblock2nd edition.

\bibitem {Teschl}\textsc{G. Teschl,}\emph{ Ordinary Differential Equations and
Dynamical Systems, }Amer. Math. Soc., 2012.

\bibitem {Walt82}\textsc{P.~Walters}, \emph{An Introduction to Ergodic
Theory}, Springer-Verlag, 1982.

\bibitem {WangHS19}\textsc{T.~Wang, Y.~Huang, and H.-W. Sun},
\emph{Measure-theoretic invariance entropy for control systems}, SIAM J.
Control Optim., 57 (2019), pp.~310--333.

\bibitem {ZF18}\textsc{M. Zanon and T. Faulwasser,} \emph{Economic MPC without
terminal constraints: Gradient-correcting end penalties enforce asymptotic
stability,} J. Process Control 63 (2018), pp. 1-14.

\bibitem {CRC}\textsc{D. Zwillinger,}\emph{ CRC\ Standard Mathematical Tables
and Formulae, }Chapman \& Hall/CRC, 2003. \newblock31st edition.
\end{thebibliography}
\end{document}